%% file: main.tex
\documentclass[a4paper]{article}
\usepackage{my_package}

\title{Properties of the gradient squared of the discrete Gaussian free field}
\date{\today}
\author[1]{Alessandra Cipriani}
\author[2]{Rajat S. Hazra}
\author[3]{Alan Rapoport}
\author[3]{Wioletta M. Ruszel}

\affil[1]{UCL, Department of Statistical Sciences, 1-19 Torrington Place, London, WC1E 7HB, UK\\
\url{a.cipriani@ucl.ac.uk}}
\affil[2]{Mathematical Institute, Leiden University, Niels Bohrweg 1, 2333 CA, Leiden, The Netherlands\\
\url{r.s.hazra@math.leidenuniv.nl}}
\affil[3]{Utrecht University, Budapestlaan 6, 3584 CD Utrecht, The Netherlands\\
\url{a.rapoport@uu.nl, w.m.ruszel@uu.nl}}

\begin{document}

\maketitle

\begin{abstract}

    In this paper we study the properties of the centered (norm of the) gradient squared of the discrete Gaussian free field in $U_{\eps}=U/\eps\cap \mathbb{Z}^d$, $U\subset \R^d$ and $d\geq 2$. The covariance structure of the field is a function of the transfer current matrix and this relates the model to a class of systems (e.g. height-one field of the Abelian sandpile model or pattern fields in dimer models) that have a Gaussian limit due to the rapid decay of the transfer current. Indeed, we prove that the properly rescaled field converges to white noise in an appropriate local Besov-H\"older space. Moreover, under a different rescaling, we determine the $k$-point correlation function and joint cumulants on $U_{\eps}$ and in the continuum limit as $\eps\to 0$. This result is related to the analogue limit for the height-one field of the Abelian sandpile (\citet{durre}), with the same conformally covariant property in $d=2$.
\end{abstract}

\input{Revision_Round_1/intro2.tex}
\input{Revision_Round_1/preliminaries.tex}
\input{Revision_Round_1/results.tex}

\input{Revision_Round_1/previous_results.tex}

\input{Revision_Round_1/cumulants.tex}
\input{Revision_Round_1/tightness2.tex}

\input{Revision_Round_1/wn.tex}
\input{Revision_Round_1/discussion.tex}
\appendix
\input{Revision_Round_1/appendix.tex}

%%%%Start biblio
\bibliographystyle{abbrvnat}
\bibliography{references}

\end{document}

%% file: Revision_Round_1/intro2.tex
\section{Introduction}

The Gaussian free field (GFF) is one of the most prominent models for random surfaces. It appears as scaling limit of observables in many interacting particle systems, see for example \citet{Kenyon,Wilson2011XORIsingLA,jerison,Sheffield2007Nov}. It serves as a building block for defining the Liouville measure in Liouville quantum gravity (see \citet{Ding2021Sep} and references therein for a list of works on the topic). 

Its discrete counterpart, the discrete Gaussian free field (DGFF), is also very well-known among random interface models on graphs. Given a graph $\Lambda$, a (random) interface model is defined as a collection of (random) real heights $\Gamma=(\Gamma(x))_{x\in \Lambda}$, measuring the vertical distance between the interface and the set of points of $\Lambda$ (\citet{funaki,Velenik2006Jan}). The discrete Gaussian free field has attracted a lot of attention due to its links to random walks, cover times of graphs, and conformally invariant processes (see \citet{Glimm, Barlow,Sheffield2007Nov,Ding2012,Schramm2009Jan}, among others). In the present paper, we will consider the DGFF on the square lattice, that is, we will focus on $\Lambda\subseteq \Z^d$, in which case the probability measure of the DGFF is a Gibbs measure with formal Hamiltonian given by
\begin{equation}\label{eq:Hamiltonian}
    H(\Gamma) = \frac{1}{2d}\sum_{x,y: \|x-y\|=1} V\left(\Gamma(x)-\Gamma(y)\right) \ ,
\end{equation}
where $V(\varphi)=\varphi^2/2$. We will always work with $0$-boundary conditions, meaning that we will set $\Gamma(x)$ to be zero almost surely outside $\Lambda$. For general potentials $V(\cdot)$ the Hamiltonian~\eqref{eq:Hamiltonian} defines a broad class of gradient interfaces which have been widely studied in terms of decay of correlations and scaling limits~(\citet{NadafSpencer,BiskupSpohn,cotar2009strict}), among others.

The gradient Gaussian free field $\nabla \Gamma$ is defined as the gradient of the DGFF $\Gamma$ along edges of the square lattice. %where $\Gamma$ is a DGFF on $\Lambda \subset \Z^d$ with $0$-boundary conditions outside $\Lambda$.
This field is a centered Gaussian process whose correlation structure can be written in terms of $T(\cdot,\cdot)$, the transfer current (or transfer impedance) matrix (\citet{kassel-wu}). Namely, if we consider the gradient $\nabla_i \Gamma(\cdot)\coloneqq\Gamma(\cdot+e_i)-\Gamma(\cdot)$ in the $i$-th coordinate direction of $\R^d$, we have, for $x,y\in\Z^d,1\le i,j\le d $, that
\[
\begin{split}
    \mathbb{E}\left[\nabla_i \Gamma(x) \nabla_j \Gamma(y)\right] & = G_{\Lambda}(x,y) - G_{\Lambda}(x+e_i,y) -G_{\Lambda}(x,y+e_j) + G_{\Lambda}(x+e_i,y+e_j)  \\
    &= T(e,f)
\end{split}
\]
where  $e=(x,x+e_i)$ and $f=(y,y+e_j)$ are directed edges of the grid and $G_{\Lambda}(\cdot,\cdot)$ is the discrete harmonic Green's function on $\Lambda$ with $0$-boundary conditions outside $\Lambda$.
Here $T(e,f)$ describes a {current flow} between $e$ and $f$.

The main object we will study in our article is the following. Take $U$ to be a connected, bounded subset of $\R^d$ with smooth boundary. Consider the recentered squared norm of the gradient DGFF, formally denoted by $$\Phi_\eps(\cdot)=\,:\!\|\nabla \Gamma\|^2\!:\!(\cdot) =\sum_{i=1}^d : \left(\Gamma(\cdot+e_i)-\Gamma(\cdot)\right)^2 :$$
on the discretized domain $U_{\eps} = U/\eps \cap \mathbb{Z}^d$, $\eps>0$, $d\geq 2$, with $\Gamma$ a $0$-boundary DGFF on $U_\eps$. The colon $:(\cdot) :$ denotes the Wick centering of the random variables. In the rest of the paper we will simply call $\Phi_\eps$ {\em the gradient squared of the DGFF}. Let us remark that we do not consider $d=1$ here since in one dimension the gradient of the DGFF is a collection of i.i.d. Gaussian variables.

\subsection*{\texorpdfstring{$k$}{k}-point correlation functions}

Our first main result  determines the $k$-point correlation functions for the field $\Phi_{\eps}$ on the discretized domain $U_\eps$ and in the scaling limit as $\eps\to 0$. We defer the precise statement to Theorem~\ref{thm:cumulants} in Section \ref{sec:main}, which we will now expose in a more informal way. Let $\eps>0$ and $k \in \mathbb{N}$ and let the points $x^{(1)},\dots,x^{(k)}$ in $U \subset \R^d$, $d\geq 2$, be given. Define $x^{(j)}_\eps$ to be a discrete approximation of $x^{(j)}$ in $U_\eps$, for $j = 1,\dots,k$. Let $\Pi([k])$ be the set of partitions of $k$ objects and $S_\mathrm{cycl}^0(B) $ be the set of cyclic permutations of a set $B$ without fixed points. Finally let $\mathcal E$ be the set of coordinate vectors of $\R^d$. Then the $k$-point correlation function at fixed ``level'' $\eps$ is equal to
\begin{equation}\label{eq:cumulantsIntro}
    \E\left[\prod_{j=1}^k{\Phi_\eps\big(x^{(j)}_\eps\big)}\right] = \sum_{\pi\in\Pi([k])}\prod_{B\in\pi}2^{|B|-1}\sum_{\sigma\in S_\mathrm{cycl}^0(B)} \sum_{\eta:B\to\mathcal{E}}\prod_{j\in B} \nabla_{\eta(j)}^{(1)}\nabla_{\eta(\sigma(j))}^{(2)}G_{U_\eps}\big(x^{(j)}_\eps,x^{(\sigma(j))}_\eps\big).
\end{equation}
Moreover if $x^{(i)} \neq x^{(j)}$ for all $i \neq j$, the scaling limit of the above expression is
\begin{multline}\label{eq:moments_limitIntro}
    \lim_{\eps\to 0} \eps^{-dk}\E\left[\prod_{j=1}^k{\Phi_\eps\big(x^{(j)}_\eps\big)}\right] = \sum_{\pi\in\Pi([k])}\prod_{B\in\pi}2^{|B|-1}
    \sum_{\sigma\in S_\mathrm{cycl}^0(B)} \\
    \sum_{\eta:B\to\mathcal{E}}\prod_{j\in B} \partial_{\eta(j)}^{(1)}\partial_{\eta(\sigma(j))}^{(2)}G_U \big(x^{(j)}, x^{(\sigma(j))}\big) \,
\end{multline}
where $G_U(\cdot,\cdot)$ is the continuum Dirichlet harmonic Green's function on $U$. As a corollary (Corollary \ref{cor:cumulants}) we also determine the corresponding cumulants on $U_\eps$ and in the scaling limit.

Let us discuss some interesting observations in the sequel. The $k$-point correlation function of~\eqref{eq:moments_limitIntro} has similarities to the $k$-point correlation that arises in permanental processes, see \citet{Last, Eisen, manju} for relevant literature. In fact, in $d=1$ one can show that the gradient squared is {\em exactly} a permanental process with kernel given by the diagonal matrix whose non-zero entries are the double derivatives of $G_U$~\cite[Theorem 1]{mccullagh2006permanental}. In higher dimensions however we cannot identify a permanental process arising from the scaling limit, since the directions of derivations of the DGFF at each point are not independent. Nevertheless the 2-point correlation functions of $\Phi_\eps$ are positive (see Equation \eqref{eq:2point} in Section \ref{sec:proofs}), which is consistent with attractiveness of permanental processes \citep[Remark on p. 139]{Last}, and the overall structure resembles closely that of permanental processes marginals.

In $d=2$, the limiting $k$-joint cumulants of first order $\kappa$ of our field are interestingly connected to the cumulants of the height-one field  $\big(h_{\eps}(x_{\eps}^{(i)}):\ x_\eps^{(i)}\in U_{\eps}\big)$ of the Abelian sandpile model~\citep[Theorem 2]{durre}. Theorem~\ref{thm:cumulants} will imply that for every set of $\ell\ge 2$ pairwise distinct points in $d=2$ one has
\begin{equation}\label{eq:equal_cumulants}
   -2\ \lim_{\eps\to 0} \eps^{-2 \ell}\kappa\left(\frac{C}{4}\Phi_\eps\big(x^{(1)}_\eps \big), \dots, \frac{C}{4}\Phi_\eps\big(x^{(\ell)}_\eps\big)\right) = \lim_{\eps\to 0} \eps^{-2 \ell}\kappa\left(h_\eps\big(x_{\eps}^{(1)} \big), \dots, h_\eps\big(x_{\eps}^{(\ell)}\big)\right)
\end{equation}
with
\begin{equation}\label{eq:def_C}C=\frac{2}{\pi} -\frac{4}{\pi^2} =\pi \ \mathbb{E}\left[h_{0}(0)\right],\end{equation}
see \citet[Theorem 6]{durrethesis}.

We would also like to point out that the apparently intricate structure of Equations~\eqref{eq:cumulantsIntro}--\eqref{eq:moments_limitIntro} and of D\"urre's Theorem 2 can be unfolded as soon as one recognizes therein the structure of a Fock space. We will discuss this point in more detail in Subsection~\ref{subsec:fock}, where in particular in Corollary~\ref{cor:cum_in_height} we will derive a Fock space representation of the $k$-point function for the height-one field. We will pose further questions on this matter in the Discussion Section~\ref{sec:dis}.

Due to the similar nature of the cumulants in the height-one field of the sandpile and our field, we show in Proposition \ref{thm:conformal} that in $d=2$ the $k$-point correlation functions are conformally covariant (compare~\citet[Theorem 1]{durre},~\citet[Theorem 2]{kassel-wu}). This hints at Theorems 2 and 3 of \citet{kassel-wu}, in which the authors prove that for finite weighted graphs the rescaled correlations of the spanning tree model and minimal subconfigurations of the Abelian sandpile have a universal and conformally covariant limit.

\subsection*{Scaling limit}

The second main result of our paper is the scaling limit of the field towards white noise in some appropriate local Besov-H\"older space. As we will show in Theorem~\ref{thm:goes_to_WN}, Section \ref{sec:main}, as $\eps\to0$ the gradient squared of the discrete Gaussian free field $\Phi_{\eps}$ converges as a random distribution to spatial white noise $W$:
\begin{equation}\label{eq:wn_intro}
         \frac{\eps^{-d/2}}{\sqrt{\chi}} \Phi_{\eps} \overset{d}\longrightarrow W \ ,
\end{equation}
for some explicit constant $0<\chi<\infty$. The result is sharp in the sense that we obtain convergence in the smallest H\"older space where white noise lives. The constant $\chi$, defined explicitly in~\eqref{eq:chi_def}, is the analogue of the susceptibility for the Ising model, in that it is a sum of all the covariances between the origin and any other lattice point.
We will prove that this constant is finite and the field $\Phi_\eps$ has a Gaussian limit. 
Note that \citet{Newman1980Jan} proves the same result
for translation-invariant fields with finite susceptibility satisfying the FKG inequality. In our case we do not have translation invariance since we work on a domain, so we are not able to apply directly this criterion. From a broader perspective there are several other results in the literature that obtain white noise in the limit due to an algebraic decay of the correlations, see for example~\citet{bauerschmidt2014scaling}.

Note that our field can be understood in a wider class of models having correlations which depend on the transfer current matrix $T(\cdot,\cdot)$. An interesting point mentioned in \citet{kassel-wu} is that pattern fields of determinantal processes closely connected to the spanning tree measure and $T(\cdot,\cdot)$ (for example the spanning unicycle, the Abelian sandpile model (\citet{durre}) and the dimer model (\citet{Bouti})) 
have a universal Gaussian limit when viewed as random distributions. Correlations of those pattern fields can be expressed in terms of transfer current matrices which decay sufficiently fast and assure the central limit-type behaviour which we also obtain.

Let us comment finally on the differences between expressions \eqref{eq:moments_limitIntro} and~\eqref{eq:wn_intro}. The scaling factors are different, and this reflects two viewpoints one can have on $\Phi_\eps$: the one of \eqref{eq:moments_limitIntro} is that of correlation functionals in a Fock space, while in \eqref{eq:wn_intro} we are looking at it as a Gaussian distributional field (compare also Theorems 2 and 3 in~\citet{durre}). This is compatible, as there are examples of trivial correlation functionals which are non-zero as random distributions~(\citet{kang2013gaussian}).

The novelty of the paper lies in the fact that we construct the gradient squared of the Gaussian free field on a grid, determine its $k$-point correlation function and scaling limits. We determine tightness in optimal Besov-H\"older spaces (optimal in the sense that we cannot achieve a better regularity for the scaling limit to hold). Furthermore we show the ``dual'' behavior in the scaling limit of the gradient squared of the DGFF as a Fock space field and as a random distribution. As mentioned before we recognize a similarity to permanental processes, and it is worthwhile noticing that for general point processes there is a Fock space structure, see e.g. \citet[Section 18]{Last}. Since there is a close connection to the height-one field via correlation structures, we also unveil a Fock space structure in the Abelian sandpile model.

\subsection*{Proof ideas}

The main idea for the proof of results \eqref{eq:cumulantsIntro}--\eqref{eq:moments_limitIntro} is to decompose the $k$-point correlation function in terms of complete Feynman diagrams~(\citet{janson}). Then we can use the Gaussian nature of the field to expand the products of covariances as transfer currents. To determine the scaling limit we will use developments from \citet{funaki} and \citet{kassel-wu}. Let us stress that the proof of the scaling limit of cumulants differs from the one of \citet[Theorem 2]{durre} who instead uses the correspondence between the height-one field and the spanning tree explicitly to determine the limiting observables.

The proof of the scaling limit~\eqref{eq:wn_intro} is divided into two parts. In a first step (Proposition~\ref{thm:tightness}) we prove that the family of fields under consideration is tight in an appropriate local Besov-H\"older space by using a tightness criterion of \citet{mf}. The proof requires a precise control of the summability of $k$-point functions, which is provided by Theorem~\ref{thm:cumulants} and explicit estimates for double derivatives of the Green's function in a domain. Observe that, even if the proof relies on the knowledge of the joint moments of the family of fields, we only use asymptotic bounds derived from them. More specifically, we need to control the rate of growth of sums of moments at different points. The second step (Proposition~\ref{prop:convWN}) consists in determining the finite-dimensional distributions and identifying the limiting field. We will first show that the limiting distribution, when tested against test functions, has vanishing cumulants of order higher or equal to three, and secondly that the limiting covariance structure is the $L^2(U)$ inner product of the test functions. This will imply that the finite-dimensional distributions of converge to those corresponding to $d$-dimensional white noise. For this we rely on generalized bounds on double gradients of the Green's function from \citet{lawlerlimic} and \citet{durre}.

\paragraph{Structure of the paper} The structure of the paper is as follows. In Section \ref{sec:not} we fix notation, introduce the fields that we study and provide the definition of the local Besov-H\"older spaces where convergence takes place. Section \ref{sec:main} is devoted to stating the main results in a more precise manner. The subsequent Section \ref{sec:proofs} contains all proofs and finally in Section \ref{sec:dis} we discuss possible generalizations and pose open questions.

%% file: Revision_Round_1/preliminaries.tex
\section{Notation and preliminaries}\label{sec:not}

\paragraph{Notation} Let $f,g$ be two functions 
$f,g:\, \R^d\to\R^d$, $d\ge 2$.  We will use $f(x)\lesssim g(x)$ to indicate that there exists a constant $C>0$ such that $|f(x)|\le C |g(x)|$, where $|\cdot|$ denotes the Euclidean norm in $\R^d$. If we want to emphasize the dependence of $C$ on some parameter (for example $U$, $\eps$) we will write $\lesssim_U$, $\lesssim_\eps$ and so on. We use the Landau symbol $f=\mathcal O(g)$ if there exist $x_0 \in \R^d$ and $C>0$ such that $|f(x)|\le C|g(x)|$ for all $x\ge x_0$. Similarly $f=o(g) $ means that $\lim_{x\to 0}f(x)/g(x)=0$. 
Furthermore, call $[\ell] \coloneqq \{1,2,\ldots,\ell\}$ and $\llbracket-\ell,\ell\rrbracket \coloneqq \{-\ell,\dots,-1,0,1,\dots,\ell\}$, for some $\ell\in \N$. 

We will write $|A|$ for the cardinality of a set $A$. 
For any finite set $A$ we define $\Pi(A)$ as the set of all partitions of $A$. Let $\mathrm{Perm}(A)$ denote the set of all possible permutations of the set $A$ (that is, bijections of $A$ onto itself). When $A = [k]$ for some $k\in\N$, we might also refer to its set of permutations as $S_k$. If we restrict $S_k$ to those permutations without fixed points, we denote them as $S^0_k$.
Call $S_\mathrm{cycl}(A)$ the set of the full cyclic permutations of $A$, possibly with fixed points. More explicitly, any $\sigma:A\to A$ bijective is in $S_\mathrm{cycl}(A)$ if $\sigma(A') \neq A'$ for any subset $A'\subsetneq A$ with $\abs{A'} > 1$. When this condition is relaxed to all $A'$ with $\abs{A'}>0$ we obtain the set of all cyclic permutations without fixed points which is called $S^0_\mathrm{cycl}(A)$. 

Let $n\in\N$ and $\mathbf X=(X_{i})_{i=1}^n$ be a vector of real-valued random variables, each of which has all finite moments.

\begin{definition}[Joint cumulants of random vector]\label{def:cum_mult}
The cumulant generating function $K(\mathbf t)$ of $\mathbf X$ for $\mathbf t=(t_1,\dots,t_n)\in \R^n$ is defined as 
\[
K(\mathbf t) \coloneqq \log \left ( \mathbb{E}\big[e^{ \mathbf t\cdot \mathbf X}\big]\right ) = \sum_{\mathbf m\in\N^n} \kappa_{\mathbf m}(\mathbf X) \prod_{j=1}^n \frac{t_j^{m_j}}{m_j!} \ ,
\]
where $\mathbf t\cdot \mathbf X$ denotes the scalar product in $\R^n$, $\mathbf m=(m_1,\dots,m_n)\in\N^n$ is a multi-index with $n$ components, and
\[
\kappa_{\mathbf m}(\mathbf X)=\frac{\partial^{|m|}}{\partial t_1^{m_1}\cdots \partial t_n^{m_n}}K(\mathbf t)\Big|_{t_1=\ldots=t_n=0} \ ,
\]
being $|m|=m_1+\cdots+m_n$. The joint cumulant of the components of $\mathbf X$ can be defined as a Taylor coefficient of $K(t_1,\ldots,t_n)$ for $\mathbf m=(1,\,\,\ldots,\,1)$; in other words
\[
\kappa(X_1,\ldots,X_n)=\frac{\partial^n}{\partial t_1\cdots \partial t_n} K(\mathbf t)\Big|_{t_1=\ldots=t_n=0} \ .
\]
In particular, for any $A\subseteq [n]$, the joint cumulant $\kappa(X_i: i\in A)$ of $\mathbf X$ can be computed as 
\[
\kappa(X_i: i\in A) = \sum_{\pi \in \Pi(A)} (|\pi|-1)! (-1)^{|\pi|-1} \prod_{B\in \pi} \mathbb{E} \left[\prod_{i\in B} X_i \right] \ ,
\]
with $|\pi|$ the cardinality of $\pi$.
\end{definition}
Let us remark that, by some straightforward combinatorics, it follows from the previous definition that
\begin{equation}\label{eq:mom_to_cum}
\mathbb{E} \left[\prod_{i\in A} X_i \right] = \sum_{\pi \in \Pi(A)} \prod_{B\in \pi} \kappa(X_i: i\in B) \ .
\end{equation}
If $A=\{i,j\}$, $i,j\in[n]$, then the joint cumulant $\kappa(X_i,X_j)$ is the covariance between $X_{i}$ and $X_{j}$. We stress that, for a real-valued random variable $X$, one has the equality
\[
\kappa(\underbrace{X,\ldots,X}_{n\text{ times}})=\kappa_n(X) \ ,\quad n\in \N \ ,
\]
which we call the \emph{$n$-th cumulant of $X$}.

\subsection{Functions of the Gaussian free field and white noise}\label{sec:GFF} 

Let \mbox{$U\subset\R^d$}, $d\geq 2$, be a non-empty bounded connected open set with $\mathcal C^1$ boundary. Denote by $(U_{\eps}, E_{\eps})$ the graph with vertex set $U_\eps \coloneqq U/\eps\cap\Z^d$ and edge set $E_\eps$ defined as the bonds induced  by the hypercubic lattice $\Z^d$ on $U_\eps$. For an (oriented) edge $e \in E_{\eps}$ of the graph, we denote by $e^+$ its tip and $e^-$ its tail, and write the edge as $e = (e^-, e^+)$. Consider $\mathcal{E} \coloneqq \{e_i\}_{\,1\leq i\leq d}$, the canonical basis of $\R^d$. Since we will use approximations via grid points, we need to introduce, for any $t\in \R^d$, its floor function as
\[
    \floor{t} \coloneqq \text{the unique }z\in\Z^d\text{ such that }t\in z+[0,1)^d \ .
\]

\begin{definition}[Discrete Laplacian on a graph]
We define the (normalized) \emph{discrete Laplacian} with respect to a vertex set $V\subseteq \Z^d$ as
\begin{equation}\label{eq:laplacian}
    \Delta_V(x,y) \coloneqq
    \begin{dcases}
    \hfil -1 & \text{if } x=y \ ,\\
    \hfil \frac{1}{2d} & \text{if } x \sim y \ ,\\
    \hfil 0 & \text{otherwise} \ .
    \end{dcases}
\end{equation}
where $x,y \in V$ and $x \sim y$ denotes that $x$ and $y$ are nearest neighbors. For any function $f:V \to \mathcal \R$ we define 
\begin{equation}\label{eq:laplacian_on_function}
	\Delta_V f(x) \coloneqq 
	\sum_{y \in V}\Delta_V(x,y) f(y) =
	\frac1{2d} \sum_{y \sim x} (f(y)-f(x)) \ , \quad x \in V \ .
\end{equation}
\end{definition}

\noindent
Call the outer boundary of $V$ as
\[
\partial^{\mathrm{ex}}V \coloneqq \{x\in\Z^d\setminus V:\,\exists \,y\in V:\,x\sim y\} \ .
\]
\begin{definition}[Discrete Green's function]\label{def:dGF}
The Green's function $G_V(x,\cdot) : V \cup \partial^\text{ex}V \to \R$, for $x \in V$, with Dirichlet boundary conditions is defined as the solution of
\[
\begin{dcases}
-\Delta_{V} G_{V}(x,y) = \delta_x(y) &\text{ if } y\in V \ , \\
\hfil G_{\Lambda}(x,y) = 0 &\text{ if } y\in \partial^{\mathrm{ex}} V \ ,
\end{dcases}
\]
where $\delta$ is the Dirac delta function.
\end{definition}

\begin{remark}
    When $V = \Z^d$, we ask for the extra condition $G_V(x,y) \to 0$ as $\|y\|\to\infty$.
\end{remark}

% For any finite set $V \subset \Z^d$ define the \emph{discrete (normalized) Laplacian} by
% \[
%     \Delta_V f(x) \coloneqq \frac{1}{2d}\sum_{y:\,|y-x|=1}\left(f(y)-f(x)\right) \ , \quad f:V\to\R \ , \ x\in\Z^d \ .
% \]

% \begin{definition}[Discrete Green's function]\label{def:dGF}
% The discrete Green's function on $V$ is the map $G_V:V\times V \to \R$ which is equal to the inverse of $-\Delta_V$ with $0$-boundary conditions, i.e. for fixed $y \in V$
% \[
%     \begin{dcases}
%         \Delta_V G_V(x,y) =- \delta_y(x) \ &\text{if} \  x \in V \ ,\\
%         G_V(x,y) = 0 \ &\text{if} \  x\notin V\ ,
%     \end{dcases}
% \]
% where $\delta$ is the Dirac delta function.
% \end{definition}

Denote by $G_0(\cdot,\cdot)$ the Green's function for the whole grid $\Z^d$ when $d\geq3$, or with a slight abuse of notation the potential kernel for $d=2$. This abuse of notation is motivated by the fact that we will only be interested in the discrete differences of $G_0$, which exist for the infinite-volume grid in any dimension. Notice that $G_0(\cdot,\cdot)$ is translation invariant; that is, $G_0(x,y) = G_0(0,y-x)$ for all $x,y\in\Z^d$.  
\begin{definition}[Continuum Green's function]\label{def:cGFF}
The {continuum Green's function} $G_U$ on $\overline U \subset \R^d$ is the solution (in the sense of distributions) of 
\begin{equation}\label{def:cGF}
    \begin{dcases}
        \Delta G_U(\cdot, y) = -\delta_y(\cdot) \ &\text{on } U \ ,\\
        G_U(\cdot,y) = 0 \ &\text{on } \partial U
    \end{dcases}
\end{equation}
for $y\in U$, where $\Delta$ denotes the continuum Laplacian and $\overline U$ is the closure of $U$.
\end{definition}
For an exhaustive treatment on Green's functions we refer to \citet{Evans2010,lawlerlimic} and \citet{spitzer}. 

\begin{definition}[Discrete Gaussian free field, {\citet[Section 2.1]{sznitman2012topics}}]\label{def:dGFF}
Let \mbox{$\Lambda\subset \Z^d$} be finite. The discrete Gaussian free field (DGFF) $(\Gamma(x))_{x\in \Lambda}$ with $0$-boundary condition is defined as the (unique) centered Gaussian field with covariance given by
\[
    \E\left[\Gamma(x) \Gamma(y)\right] = G_\Lambda(x,y) \ , \quad x,y \in \Lambda \ .
\]
\end{definition}
Define for an oriented edge $e = (e^-, e^+) \in E_\eps$ the {\it gradient DGFF $\nabla_e\Gamma$} as
\begin{equation}
    \nabla_e\Gamma(e^-) \coloneqq \Gamma(e^+) - \Gamma(e^-) \ .
\end{equation}
In the following, we will define the main object of interest.

\begin{definition}[Gradient squared of the DGFF]\label{def:GradGFF}
The discrete stochastic field $\Phi_\eps$ given by
\[
   \Phi_\eps(x) \coloneqq \sum_{i=1}^d{:\!\left({\nabla_{e_i}\Gamma}(x)\right)^2\!:} \ , \quad x \in U_\eps \ ,
\]
is called the gradient squared of the DGFF, where $: \cdot :$ denotes the Wick product; that is, $:\!X\!:\, = X - \E[X]$ for any random variable $X$.
\end{definition}

The family of random fields $\left(\Phi_\eps\right)_{\eps>0}$ is a family of distributions, which is defined to act on a given test function $f\in\mathcal C_c^\infty(U)$ as
\begin{equation}\label{def:action}
    \langle\Phi_\eps,f\rangle \coloneqq \int_U {\Phi_\eps\left(\floor{x/\eps}\right) f(x) \, \d x} \ ,
\end{equation}
where we take $\Phi_\eps(\floor{x/\eps}) = 0$ in case $\floor{x/\eps} \notin U_\eps$, which can happen if $\eps$ is not small enough.
% where $\langle \cdot, \cdot \rangle$ denotes the scalar product in $L^2(U)$.

When no ambiguities appear, we will write $\left(\nabla_i\Gamma(x)\right)^2$ for $\left(\nabla_{e_i}\Gamma(x)\right)^2$, with $i=1,2,\dots,d$. 
For any given function $f:\Z^d\times\Z^d\to\R$, we define the {\it discrete gradient} in the first argument and direction $e_i\in\mathcal E$ as
\[
    \nabla_{e_i}^{(1)}f(x,y) \coloneqq f(x+e_i,y) - f(x,y) ,
\]
with $x,y\in\Z^d$, and analogously
\[
    \nabla_{e_i}^{(2)}f(x,y) \coloneqq f(x,y+e_i) - f(x,y)
\]
for the second argument. Once again, when no ambiguities arise, we will write $\nabla_i^{(1)}$ for $\nabla_{e_i}^{(1)}$, and analogously for the second argument. 

For a continuum function $g:U\times U\to\R$, $\partial_{e_i}^{(1)}g(x,y)$ denotes the partial derivative of $g$ with respect to the first argument in the direction of $e_i$, while $\partial_{e_i}^{(2)}g(x,y)$ corresponds to the second argument, also in the direction $e_i$. The same abuse of notation on the subindex $e_i$ applies here.

\begin{definition}[Gaussian white noise]\label{def:WN}
The $d$-dimensional Gaussian white noise $W$ is the centered Gaussian random distribution on $U \subset \R^d$ such that, for every $f,g\in L^2(U)$, 
\[
\mathbb{E}\left[\langle W,f \rangle \langle W,g \rangle\right] = \int_{U} f(x) g(x) \d x  \ .
\]
In other words, $\langle W,f \rangle \sim \mathcal N \big(0, \|f\|^2_{L^2(U)} \big)$ for every $f\in L^2(U)$. 
\end{definition}

\subsection{Besov-H\"older spaces}

In this Subsection we will define the functional space on which convergence will take place. We will use \citet{mf} as a main reference. Local H\"older and Besov spaces of negative regularity on general domains are natural functional spaces when considering scaling limits of certain random distributions or in the context of non-linear stochastic PDE's, see e.g. \citet{mf,Hairer} especially when those objects are well-defined on a domain $U \subset \R^d$ but not necessarily on the full space $\R^d$. They are particularly suited for fields which show bad behaviour near the boundary $\partial U$.

Let $(V_n)_{n\in \mathbb{Z}}$ be a dense subsequence of subspaces of $L^2(\R^d)$ such that $\bigcap_{n\in \mathbb{Z}} V_n = \{0\}$. Denote by $W_n$ the orthogonal complement of $V_n$ in $V_{n+1}$ for all $n\in \mathbb{Z}$. Furthermore, we assume the following properties. The function $f\in V_n$ if and only if $f(2^{-n} \cdot) \in V_0$. Let $(\phi(\cdot -k))_{k\in \mathbb{Z}^d}$ be an orthonormal basis of $V_0$ and $(\psi^{(i)}(\cdot -k))_{i<2^d, k\in \mathbb{Z}^d}$ an orthonormal basis of $W_0$. Note that $\phi, (\psi^{(i)})_{i<2^d}$ both belong to $\mathcal{C}^r_c(\R^d)$ for some positive integer $r \in \mathbb{N}$, that is, they belong to the set of $r$ times continuously differentiable functions on $\R^d$ with compact support. For more details about wavelet analysis, see~\cite{daubechies1992ten,meyer1992wavelets}.

Define $\Lambda_n=\mathbb{Z}^d/2^n$ and
\[
\phi_{n,x}(y) = 2^{dn/2}\phi\big(2^n(y-x)\big)
\]
resp.
\[
\psi^{(i)}_{n,x}(y) = 2^{dn/2}\psi^{(i)}\big(2^n(y-x)\big)
\]
which makes $(\phi_{n,x})_{x\in \Lambda_n}$ an orthonormal basis of $V_n$ resp. $(\psi^{(i)}_{n,x})_{x\in \Lambda_n, i<2^d, n\in \mathbb{Z}}$ an orthonormal basis of $L^2(\R^d)$. Every function $f\in L^2(\R^d)$ can be decomposed into
\[
f = \mathcal{V}_k f + \sum_{n=k}^{\infty} \mathcal{W}_n f
\]
for any fixed $k\in \mathbb{Z}$, where $\mathcal{V}_n$ resp. $\mathcal{W}_n$ are the orthogonal projections onto $V_n$ resp. $W_n$ defined as
\[
\mathcal{V}_n f = \sum_{x\in \Lambda_n} \big\langle f,\phi_{n,x}\big\rangle \phi_{n,x} \ , \quad \mathcal{W}_n f = \sum_{i<2^d, x\in \Lambda_n} \big\langle f,\psi^{(i)}_{n,x}\big\rangle \psi^{(i)}_{n,x} \ .
\]

\begin{definition}[Besov spaces]\label{def:locBes}
Let $\alpha \in \R$, $|\alpha|<r$, $p,q\in [1,\infty]$ and $U\subset \R^d$. The Besov space $\mathcal{B}^{\alpha}_{p,q}(U)$ is the completion of $\mathcal{C}^{\infty}_c(U)$ with respect to the norm
\[
\|f\|_{\mathcal{B}^{\alpha}_{p,q}} \coloneqq \|\mathcal{V}_0 f\|_{L^p} + \left\| \big(2^{\alpha n} \|\mathcal{W}_n f\|_{L^p}\big)_{n\in \mathbb{N}}\right\|_{\ell^q} \ .
\]
The local Besov space $\mathcal{B}^{\alpha, \mathrm{loc}}_{p,q}(U)$ is the completion of $\mathcal{C}^{\infty}(U)$ with respect to the family of semi-norms
\[
f \mapsto \| \widetilde\chi f\|_{\mathcal{B}^{\alpha}_{p,q}}
\]
indexed by $\widetilde\chi \in \mathcal{C}^{\infty}_c(U)$.
\end{definition}

We will use the following embedding property of Besov spaces in the tightness argument.

\begin{lemma}[{\citet[Remark 2.12]{mf}}]\label{lemma:embedding}
    For any $1\leq p_1 \leq p_2 \leq\infty$, $q\in[1,\infty]$ and $\alpha\in\R$, the space $\mathcal B_{p_2,q}^{\alpha,\mathrm{loc}}(U)$ is continuously embedded in $\mathcal B_{p_1,q}^{\alpha,\mathrm{loc}}(U)$.
\end{lemma}

Finally let us define the functional space where convergence will take place, the space of distributions with locally $\alpha$-H\"older regularity. For that, we denote as $\mathcal C^r$ the set of $r$ times continuously differentiable functions on $\R^d$, with $r\in\N\cup\{\infty\}$. We also define the $\mathcal C^r$ norm of a function $f\in\mathcal C^r$ as
\[
    \|f\|_{\mathcal C^r} \coloneqq \sum_{|i|\leq r} \|\partial_i f\|_{L^\infty} \ ,
\]
being $i\in\N^d$ a multi-index.

\begin{definition}[H\"older spaces]\label{def:locHoel}
Let $\alpha < 0,r_0=-\lfloor \alpha \rfloor$. The space $\mathcal C^{\alpha}_{\mathrm{loc}}(U)$ is called the locally H\"older space with regularity $\alpha \in \R$ on the domain $U$. It is the completion of $\mathcal{C}^{\infty}_c(U)$ with respect to the family of semi-norms
\[
f \mapsto \|\widetilde\chi f\|_{\mathcal{C}^{\alpha}}
\]
indexed by $\widetilde\chi \in \mathcal{C}^{\infty}_c(U)$ and
\[
\| f\|_{\mathcal C^{\alpha}} = \sup_{\lambda \in (0,1]} \sup_{x\in \R^d} \sup_{\eta \in B^{r_0}} \lambda^{-\alpha} \int_{\R^d} f(\cdot) \, \lambda^{-d} \,\eta\left(\frac{\cdot -x}{\lambda}\right) \ ,
\]
where
\[
B^{r_0} = \left\{ \eta \in \mathcal C^{r_0}:\|\eta \|_{\mathcal C^{r_0}} \leq 1, \, \supp \eta \subset B(0,1) \right\} \ .
\]
\end{definition}

Note that by \citet[Remark 2.18]{mf} one has $\mathcal C^{\alpha}_{\mathrm{loc}}(U) = \mathcal{B}^{\alpha, \mathrm{loc}}_{\infty,\infty}(U)$.

%% file: Revision_Round_1/results.tex
\section{Main results}\label{sec:main}

The first result we would like to present is an explicit computation of the $k$-point correlation function of the  gradient squared of the DGFF field $\Phi_\eps$ defined in Definition \ref{def:GradGFF}.

\begin{theorem}\label{thm:cumulants}
Let $\eps>0$ and $k \in \mathbb{N}$ and let the points $x^{(1)},\dots,x^{(k)}$ in $U \subset \R^d$, $d\geq 2$, be given. Define $x^{(j)}_\eps \coloneqq \floor{x^{(j)}/\eps}$ and choose $\eps$ small enough so that $x^{(j)}_\eps \in U_\eps$, for all $j = 1,\dots,k$. Then
\begin{equation}\label{eq:cumulants}
    \E\left[\prod_{j=1}^k{\Phi_\eps\big(x^{(j)}_\eps\big)}\right] = \sum_{\pi\in\Pi([k])}\prod_{B\in\pi}2^{|B|-1}\sum_{\sigma\in S_\mathrm{cycl}^0(B)} \sum_{\eta:B\to\mathcal{E}}\prod_{j\in B} \nabla_{\eta(j)}^{(1)}\nabla_{\eta(\sigma(j))}^{(2)}G_{U_\eps}\big(x^{(j)}_\eps,x^{(\sigma(j))}_\eps\big) \,
\end{equation}
where $G_{U_{\eps}}(\cdot,\cdot)$ was defined in Definition \ref{def:dGF}.
Moreover if $x^{(i)} \neq x^{(j)}$ for all $i \neq j$, then
\begin{multline}\label{eq:moments_limit}
    \lim_{\eps\to 0} \eps^{-dk}\E\left[\prod_{j=1}^k{\Phi_\eps\big(x^{(j)}_\eps\big)}\right] = \sum_{\pi\in\Pi([k])}\prod_{B\in\pi}2^{|B|-1}
    \sum_{\sigma\in S_\mathrm{cycl}^0(B)} \\
    \sum_{\eta:B\to\mathcal{E}}\prod_{j\in B} \partial_{\eta(j)}^{(1)}\partial_{\eta(\sigma(j))}^{(2)}G_U \big(x^{(j)}, x^{(\sigma(j))}\big) \,
\end{multline}
where $G_U(\cdot,\cdot)$ was defined in Equation \eqref{def:cGF}.
\end{theorem}

\begin{remark}
    It will sometimes be useful to write \eqref{eq:cumulants} as the equivalent expression
    \begin{equation}\label{eq:cumulants2}
        \E\left[\prod_{j=1}^k{\Phi_\eps\big(x^{(j)}_\eps\big)}\right] = \sum_{\substack{\pi\in\Pi([k])\\\text{\normalfont{w/o singletons}}}} \prod_{B\in\pi}2^{|B|-1}\sum_{\sigma\in S_\mathrm{cycl}(B)} \sum_{\eta:B\to\mathcal{E}}\prod_{j\in B} \nabla_{\eta(j)}^{(1)}\nabla_{\eta(\sigma(j))}^{(2)}G_{U_\eps}\big(x^{(j)}_\eps,x^{(\sigma(j))}_\eps\big) \ ,
    \end{equation}
    where the condition of $\sigma$ belonging to full cycles of $B$ without fixed points is inserted in the no-singleton condition of the permutations $\pi$.
\end{remark}

\begin{remark}
    From the above expression it is immediate to see that the $2$-point function is given by
    \[
        \E\left[\Phi_\eps\big(x_\eps\big) \Phi_\eps\big(y_\eps\big)\right] =2\sum_{i,j\in[d]} \left(\nabla_i^{(1)}\nabla_j^{(2)}G_{U_\eps}\big(x_\eps,y_\eps\big)\right)^2 \ ,
    \]
    which will be useful later on.
\end{remark}

The following Corollary is a direct consequence of Theorem \ref{thm:cumulants}.

\begin{corollary}\label{cor:cumulants}
Let $\ell\in\N$. The joint cumulants $\kappa\left(\Phi_\eps\big(x^{(j)}_\eps\big):j\in[\ell],\,x^{(j)}_\eps\in U_\eps\right)$ of the field $\Phi_{\eps}$ at ``level'' $\eps>0$ are given by
\begin{equation}\label{eq:cumulants3}
    \kappa\Big(\Phi_\eps\big(x^{(j)}_\eps\big):j\in[\ell]\Big) = 2^{\ell-1}\sum_{\sigma\in S_\mathrm{cycl}^0([\ell])} \sum_{\eta:[\ell]\to\mathcal{E}}\prod_{j=1}^\ell \nabla_{\eta(j)}^{(1)}\nabla_{\eta(\sigma(j))}^{(2)}G_{U_\eps}\big(x^{(j)}_\eps,x^{(\sigma(j))}_\eps\big) \ .
\end{equation}
Moreover if $x^{(i)} \neq x^{(j)}$ for all $i \neq j$, then
\begin{equation}\label{eq:limiting_cumulants}
    \lim_{\eps\to 0} \eps^{-d\ell} \kappa\Big(\Phi_\eps\big(x^{(j)}_\eps\big):j\in[\ell]\Big) = 2^{\ell-1}\sum_{\sigma\in S_\mathrm{cycl}^0([\ell])} \sum_{\eta:[\ell]\to\mathcal{E}}\prod_{j=1}^\ell \partial_{\eta(j)}^{(1)}\partial_{\eta(\sigma(j))}^{(2)}G_U\big(x^{(j)},x^{(\sigma(j))}\big) \ .
\end{equation}
\end{corollary}
As already mentioned in the introduction, comparing our result with~\citet[Theorem 2]{durre} we obtain~\eqref{eq:equal_cumulants}.
 
The following proposition states that in $d=2$ the limit of the field $\Phi_{\eps}$ is conformally covariant with scale dimension 2. This result can also be deduced for the height-one field for the sandpile model, see \citet[Theorem 1]{durre}.

\begin{prop}\label{thm:conformal}
    Let $U,U' \subset \R^2$, $k\in \mathbb{N}$, $\big\{x^{(j)}\big\}_{j\in[k]}$, and $\big\{x^{(j)}_\eps\big\}_{j\in[k]}$ be as in Theorem \ref{thm:cumulants}. Furthermore let $h:U\to U'$ be a conformal mapping and call $h_\eps\big(x^{(j)}\big) \coloneqq \floor{h\big(x^{(j)}\big)/\eps}$, for $\eps$ small enough so that $h_\eps\big(x^{(j)}\big) \in U'_\eps$ for all $j\in[k]$. Then
    \[
        \lim_{\eps\to 0} \eps^{-2k}\E\left[\prod_{j=1}^k{\Phi^U_\eps\big(x^{(j)}_\eps\big)}\right] = \prod_{j=1}^k \abs{h'\big(x^{(j)}\big)}^2 \lim_{\eps\to 0} \eps^{-2k}\E\left[\prod_{j=1}^k{\Phi^{U'}_\eps\Big(h_\eps\big(x^{(j)}\big)\Big)}\right] \ ,
    \]
    where now for clarity we emphasize the dependence of $\Phi_\eps$ on its domain.
\end{prop}
% Let us now show in which sense we can make our field converge to white noise. First of all, we way in which we previously defined the action $\langle\Phi_\eps,f\rangle$ on a given test function $f$ was \textcolor{red}{(say where I defined this (once I do it))}
% \[
%     \langle\Phi_\eps,f\rangle = \int_U {\Phi_\eps\left(\floor{x/\eps}\right) f(x) \, \d x} .
% \]
% For the sake of convenience of computations, we will define $\left\langle\Phi_\eps,f\right\rangle$ as
% \[
%     \langle \Phi_\eps,f\rangle = \sum_{v\in U_\eps} {f(\eps v)\Phi_\eps(v)} .
% \]
% \textcolor{red}{Write explicitly the computation that shows that the difference of both of this expressions is $\mathcal O(\eps^2)$. IMPORTANT: for it to be $\mathcal O(\eps^2)$, the integral should have scaling $\eps^{-1}$, and the sum $\eps$.}

% \marginpar{Write only ONE def of <,> (the integral probably). In the proof of thm 3, write a lemma first, saying that this def in integral form is equal to the sum def, up to eps2.}

Finally we will show that the rescaled gradient squared of the discrete Gaussian free field will converge to white noise in some appropriate locally H\"older space with negative regularity $\alpha$ in $d\geq 2$ dimensions. This space is denoted as $\mathcal C_\mathrm{loc}^\alpha(U)$ (see Definition \ref{def:locHoel}).

\begin{theorem}\label{thm:goes_to_WN}
    Let $U\subset \R^d$ with $d\geq 2$. The gradient squared of the discrete Gaussian free field $\Phi_{\eps}$ converges in the following sense as $\eps\to0$:
    \[
        \frac{\eps^{-d/2}}{\sqrt \chi} \Phi_{\eps} \overset{d}\longrightarrow W,
    \]
    where the white noise $W$ is defined in Definition \ref{def:WN}.
    This convergence takes place in $\mathcal C^{\alpha}_{\mathrm{loc}}(U)$ for any $\alpha < -d/2$, and the constant $\chi$ defined as
    \begin{equation}\label{eq:chi_def}
        \chi \coloneqq 2\sum_{v\in\Z^d} \sum_{i,j\in[d]} \left(\nabla_i^{(1)}\nabla_j^{(2)}G_0(0,v)\right)^2
    \end{equation}
    is well-defined, in the sense that $0<\chi<\infty$.
\end{theorem}

\begin{remark}
Let us remind the reader that $\mathcal C^{\alpha}_{\mathrm{loc}}(U)$ with $\alpha < -d/2$ are the optimal spaces in which the white noise lives. See for example \citet[Proposition 5.9]{mourratarmstrong}.
\end{remark}

%%%%%%%%%%

\subsection{Fock space structure}\label{subsec:fock}

Let us discuss in the following the connection to Fock spaces. We start by reminding the reader of the definition of the continuum Gaussian free field (GFF).

\begin{definition}[Continuum Gaussian free field, {\citet[Section 1.5]{berenotes}}]
The continuum Gaussian free field $\overline{\Gamma}$ with $0$-boundary (or Dirichlet) conditions outside $U$ is the unique centered Gaussian process indexed by $\mathcal C_c^\infty(U)$ such that
\[
    \cov{\big(\overline{\Gamma}(f),\overline{\Gamma}(g)\big)} = \int_{U\times U} f(x)g(y)G_U(x,y) \, \d x\d y \ , \quad f,g\in \mathcal C_c^\infty(U)\ , 
\]
where $G_U(\cdot, \cdot)$ was defined in Definition \ref{def:cGFF}.
\end{definition}

We can think of it as an isometry $\overline{\Gamma}: \mathcal H \to L^2(\Omega,\mathbb P)$, for some Hilbert space $\mathcal H$ and some probability space $(\Omega, \mathcal{F},\mathbb P)$. To fix ideas, throughout this Section let us fix $\mathcal H\coloneqq\mathcal H_0^1(U)$, the order one Sobolev space with Dirichlet inner product (see~\citet[Section 1.6]{berenotes}). Note that, even if the GFF is not a proper random variable, we can define its derivative as a Gaussian distributional field.
\begin{definition}[Derivatives of the GFF,~{\citet[pg. 4]{kang2013gaussian}}]
The derivative of $\overline{\Gamma}$ is defined as the Gaussian distributional field $\partial_i \overline{\Gamma}$, $1\le i\le d$, in the following sense:
\[
\big(\partial_i \overline{\Gamma}\big)(f)\coloneqq \overline{\Gamma}\left(\partial_i f\right) \ ,\quad f\in \mathcal C_c^\infty(U) \ .
\]
\end{definition}
There is however another viewpoint that one can take on the GFF and its derivatives, and is that of viewing them as {\em Fock space fields}. This approach will be used to reinterpret the meaning of Theorem~\ref{thm:cumulants}. For the reader's convenience we now recall here some basic facts about Fock spaces and their fields. Our presentation is drawn from~\citet[Section 3.1]{janson} and~\citet[Sec. 1.2-1.4]{kang2013gaussian}.

For $n\geq0$, we denote $\mathcal H^{\odot n}$ as the $n$-th symmetric tensor power of $\mathcal H$; in other words, $\mathcal H^{\odot n}$ is the completion of linear combinations of elements $f_1\odot\cdots\odot f_n$ with respect to the inner product 
\[
    \left\langle f_1\odot\cdots\odot f_n, g_1\odot\cdots\odot g_n \right\rangle = \sum_{\sigma\in S_n} \prod_{i=1}^n \big\langle f_i,g_{\sigma(i)}\big\rangle \ ,\quad f_i, g_i\in \mathcal H \ ,\ 1\le i\le n \ .
\]
The symmetric Fock space over $\mathcal H$ is
\[
    \mathrm{Fock}(\mathcal H) \coloneqq \bigoplus_{n\geq0} \mathcal{H}^{\odot n} \ .
\]

% This way, the product operation $\odot$ of the algebra $\sum_{n\geq0}\mathcal H^{\odot n}$ corresponds to a  projection given by
% \[
%     asdf
% \]

% for $X\in\mathcal H_p$, $Y\in\mathcal H_q$, where
% \[
%     \mathcal H_p \coloneqq  \ .
% \]
% We call it \emph{Wick product}, and denote it also as $:XY:$\,.
We now introduce elements in $Fock(\mathcal H)$ called {\em Fock space fields}. We call \emph{basic correlation functionals} the formal expressions of the form
\[
    \mathcal X_p = X_{1}(x_{1})\odot\cdots\odot X_{p}(x_{p}) \ ,
\]
for $p\in\N$, $x_{1},\dots,x_{p}\in U$, and $X_{1},\dots,X_{p}$ derivatives of $\overline \Gamma$. The set $\mathcal S(\mathcal X_p)\coloneqq\{x_1,\ldots,x_p\}$ is called the set of {\em nodes} of $\mathcal X_p$. Basic Fock space fields are formal expressions written as products of
derivatives of the Gaussian free field $\overline \Gamma$, for example $1\odot \overline\Gamma$, $\partial \overline\Gamma\odot\overline\Gamma\odot \overline\Gamma $ etc. A general Fock space field $X$ is a linear combination of basic fields. We think of any such $X$ as a map $u\mapsto X(u)$, $u\in U$, where the values $\mathcal X=X(u)$ are correlation functionals with $\mathcal S(\mathcal X)=\{u\}$. Thus Fock space fields are functional-valued functions. Observe that Fock space fields may or may not be distributional random fields, but in any case we can think of them as functions in $U$ whose
values are correlation functionals.

Our goal is to define now tensor products. We will restrict our attention to tensor products over an even number of correlation functionals, even if the definition can be given for an arbitrary number of them. The reason behind this presentation is due to the set-up we will be working with.
\begin{definition}[Tensor products in Fock spaces]
Let $m\in 2\N$. Given a collection of correlation functionals 
\[\mathcal X_j\coloneqq X_{j1}(z_{j1})\odot \cdots \odot X_{j n_j}(z_{j n_j}) \ ,\quad 1\le j\le m
\]
with pairwise disjoint $\mathcal S(\mathcal X_j)$'s, the tensor product of the elements $\mathcal X_1,\ldots,\mathcal X_m$ is defined as
\begin{equation}\label{eq:tensor}
    \mathcal X_1 \cdots \mathcal X_m \coloneqq \sum_\gamma \prod_{\{u,v\}\in E_\gamma} \E\left[X_u(x_u)X_v(x_v)\right] \ ,
\end{equation}
where the sum is taken over Feynman diagrams $\gamma$ with vertices $u$ labeled by functionals $X_{pq}$ in such a way that there are no contractions of vertices in the same $\mathcal S(\mathcal X_p)$. $E_\gamma$ denotes the set of edges of $\gamma$. One extends the definition of tensor product to general correlation functionals by linearity.
\end{definition}

The reader may have noticed that~\eqref{eq:tensor} is simply one version of Wick's theorem. It is indeed this formula that will allow us in Subsection~\ref{subsec:cum} to prove Theorem~\ref{thm:cumulants}, and that enables one to bridge Fock spaces and our cumulants in the following way.  For any $j\in[k]$, $k\in\N$, $i_j\in[d]$, one can define the basic Fock space field $X_{i_j} \coloneqq \partial_{i_j} \overline{\Gamma}$. Introduce the correlation functional 
\begin{equation}\label{eq:Y}
    \mathcal Y_j\coloneqq\sum_{i_j\in \mathcal E}X_{i_j}^{\odot 2}\big(x^{(j)}\big)
\end{equation}
for $x^{(j)}\in U$. We obtain now the statement of the next Lemma.
\begin{lemma}[$k$-point correlation functions as Fock space fields] Under the assumptions of Theorem~\ref{thm:cumulants},
\[
  \lim_{\eps\to 0} \eps^{-dk}\E\left[\prod_{j=1}^k{\Phi_\eps\big(x^{(j)
  }_\eps\big)}\right] = \sum_{\pi\in \Pi([k])}\left(\frac{1}{2} \right)^{|\pi|}\prod_{B\in \pi}\mathcal Y_B\big(x^{(B)}\big)
  \]
  where ${\mathcal Y}_B\big(x^{(B)}\big)\coloneqq 2{\mathcal Y}_1\odot\cdots\odot 2 {\mathcal Y}_j$, $\mathcal S(\mathcal{ Y}_j)=\{x^{(j)}\}$, $j\in B$. Here $|\pi|$ stands for the number of blocks of the partition $\pi$ and the tensor product on the r.h.s. is taken in the sense of \eqref{eq:tensor}.
\end{lemma}
%One can see that the right-hand side of~\eqref{eq:moments_limit} is nothing but the tensor product $\mathcal Y_1\cdots\mathcal Y_k$ of correlation functionals. 
The Fock space structure is more evident from the Gaussian perspective of the DGFF, but~\eqref{eq:equal_cumulants} together with D\"urre's theorem entail a corollary which we would like to highlight. We remind the reader of the definition of the constant $C$ in~\eqref{eq:def_C}.
\begin{corollary}[Height-one field $k$-point functions, $d=2$]\label{cor:cum_in_height}
With the same notation of Theorem~\ref{thm:cumulants} one has in $d=2$ that
\[
    \lim_{\eps\to 0} \eps^{-2k} \mathbb E\left[\prod_{j=1}^k{\left(h_\eps\big(x^{(j)}_\eps\big)-\mathbb E\left[h_\eps\big(x^{(j)}_\eps\big)\right]\right)}\right] =\sum_{\pi\in \Pi([k])}\left(- \frac{1}{2}\right)^{|\pi|}\prod_{B\in \pi}\widetilde{\mathcal Y}_B\big(x^{(B)}\big)
\]
where $\widetilde{\mathcal Y}_B\big(x^{(B)}\big)\coloneqq\widetilde{\mathcal Y}_1\odot\cdots\odot\widetilde{\mathcal Y}_j,\,\,\mathcal S(\widetilde{\mathcal Y}_j)=\{x^{(j)}\}$ and $\widetilde{\mathcal Y}_j\coloneqq C\,\,\mathcal Y_j$, $j\in B$. As before, $|\pi|$ stands for the number of blocks of the partition $\pi$.
\end{corollary}

\begin{remark}
Mind that our Green's functions differ from those of \cite{durre} by a factor of $2d$ since in their definitions we use the normalized Laplacian, whereas D\"urre uses the unnormalized one. This has to be accounted for when comparing the corresponding results in both papers.
\end{remark}

%% file: Revision_Round_1/previous_results.tex
\section{Proofs}\label{sec:proofs}

\subsection{Previous results from literature}

Let us now expose some important results that we will refer to throughout the proofs. They refer to partially known results and partially consist of straightforward generalizations of previous results.

Our computations will rely on the fact that the distribution of the gradient field $\nabla_i \Gamma$, \mbox{$i\in [d]$}, is well-known. The following result is quoted from \citet[Lemma 3.6]{funaki}.

\begin{lemma}\label{lem:grad_gff_cov}
Let \mbox{$\Lambda\subset \Z^d$} be finite, and let $(\Gamma_x)_{x\in \Lambda}$ be a $0$-boundary conditions DGFF on $\Lambda$ (see Definition \ref{def:dGFF}). Then
\[
\begin{dcases}
\E\left[\nabla_i \Gamma(x)\right]=0 \ &\text{if} \ x\in\Lambda\ , \ i\in[d] \ , \\
    \E\left[\nabla_i \Gamma(x)\nabla_j \Gamma(y)\right]=\nabla^{(1)}_i\nabla^{(2)}_j G_\Lambda(x,y) \ &\text{if} \ x,y\in\Lambda \ ,\ i,j\in[d] \ .
\end{dcases}
\]
\end{lemma}
Consequently, we can directly link the gradient DGFF to so-called  transfer current matrix $T(\cdot,\cdot)$ by
\begin{equation}
    T(e,f) = G_\Lambda(e^-,f^-)-G_\Lambda(e^+,f^-)-G_\Lambda(e^-,f^+)+G_\Lambda(e^+,f^+)
\end{equation}
where $e,\,f$ are oriented edges of $\Lambda$ (see \citet[Section 2]{kassel-wu}). Equivalently we can write
\begin{equation}\label{eq:transfer_cov}
    T(e,f) = \nabla_e\nabla_f G_\Lambda(e^-,f^-) \ .
\end{equation}

From Lemma~\ref{lem:grad_gff_cov}, it is clear that we need to control the behaviour of double derivatives of discrete Green's function in the limit $\varepsilon\to 0$. In order to find the limiting joint moments of the point-wise field $\Phi_\eps(x)$ we will need the following result about the convergence of the discrete difference of the Green's function on $U_\eps$ (see Definition \ref{def:dGF}) to the double derivative of the continuum Green's function $G_U(\cdot,\cdot)$ on a set $U$ (see Equation \eqref{def:cGF}). This result follows from Theorem 1 of \citet{kassel-wu}.

\begin{lemma}[\textit{Convergence of the Green's function differences}]\label{lemma:conv_green_diff}
    Let $v$, $w$ be points in the set $U$, with $v \neq w$. Then for all $a,b \in \mathcal E$,
    \[
       \lim_{\eps\to 0} \eps^{-d}\, \nabla_{a}^{(1)}\nabla_{b}^{(2)}G_{U_\eps}\left(\floor{v/\eps},\floor{w/\eps}\right) = \partial_{a}^{(1)}\partial_{b}^{(2)}G_U(v,w) \ .
    \]
\end{lemma}

The next lemma is a generalization of \citet[Lemma 31]{durrethesis} for general dimensions $d\geq 2$. The proof is straightforward and will be omitted.
It provides an error estimate when replacing the double difference of $G_{U_\eps}(\cdot,\cdot)$ on the finite set by that of $G_0(\cdot,\cdot)$ defined on the whole lattice. 

\begin{lemma}\label{lemma:bound_2_points_durre}
    Let $D\subset U$ be such that the distance between $D$ and $U$ is non-vanishing, that is, $\dist{(D,\partial U)} \coloneqq \inf_{(x,y)\in D\times\partial U}\abs{x-y} > 0$. There exist $c_D > 0$ and $\eps_D > 0$ such that, for all $\eps \in (0,\eps_D]$, for all $v,w \in D_\eps \coloneqq D/\eps \cap \Z^d$ and $i,j\in[d]$,
    \begin{equation}\label{eq:one_durre}
        \abs{\nabla_i^{(1)}\nabla_j^{(2)} G_{U_\eps}(v,w) - \nabla_i^{(1)}\nabla_j^{(2)} G_0(v,w)} \leq c_D \, \eps^d \ ,
    \end{equation}
    and also
    \begin{equation}\label{eq:two_durre}
      \abs{\nabla_i^{(1)}\nabla_j^{(2)} G_{U_\eps}(v,w)} \leq c_D \cdot
        \begin{dcases}
            \ \abs{v-w}^{-d} \ &\text{if} \ v\neq w \ ,\\
            \hfil 1 \ &\text{if} \ v = w \ .
        \end{dcases}
    \end{equation}
\end{lemma}

An immediate consequence of \eqref{eq:two_durre} and the expression \eqref{eq:cumulants} in Theorem \ref{thm:cumulants} for two points gives us the following bound on the covariance of the field:

\begin{corollary}\label{cor:covariance_bound}
    Let $D$, $v$ and $w$ be as in Lemma \ref{lemma:bound_2_points_durre}. Then
    \begin{equation}\label{eq:bound_double_diff}
        \E\left[\Phi_\eps(v)\Phi_\eps(w)\right] \leq c_D \cdot
        \begin{dcases}
            \abs{v-w}^{-2d} \ &\text{if} \ v\neq w \ ,\\
            1 \ &\text{if} \ v = w \ .
        \end{dcases}
    \end{equation}
\end{corollary}

On the other hand, we will also make use of a straightforward extension of~\citet[Corollary 4.4.5]{lawlerlimic} for $d=2$ and \citet[Corollary 4.3.3]{lawlerlimic} for $d\geq 3$, yielding the following Lemma.

\begin{lemma}[Asymptotic expansion of the Green's function differences]\label{lemma:lemma_29_durre}
As $|v|\to+\infty$, for all $i,\,j\in[d]$
\[
    \abs{\nabla_i^{(1)}\nabla_j^{(2)} G_{0}(0,v)}=\mathcal O\big(|v|^{-d}\big) \ .
\]
\end{lemma}

The following technical combinatorial estimate, which is an immediate extension of a corollary of \citet[Lemma 37]{durrethesis}, will be important when proving tightness of the family $(\Phi_\eps)_\eps$, in order to bound the rate of growth of the moments of $\langle\Phi_\eps,f\rangle$ for some test function $f$:

\begin{lemma}\label{lemma:sum_p_point_function}
    Let $D \subset U$ such that $\dist{(D,\partial U)} > 0$ and $p\geq2$. Then
    \[
        \sum_{\substack{v_1,\dots,v_p\in D_\eps\\v_i\neq v_j \mathrm{ for }\, i\neq j}} \left(\prod_{i=1}^{p-1} \frac{1}{\left|v_i-v_{i+1}\right|^d}\right) \frac{1}{\left|v_p-v_1\right|^d} = \mathcal O_{D}\left(\eps^{-\frac{p}{2}-d+1}\right) \ ,
    \]
    where $D_\eps \coloneqq D/\eps\cap\Z^d$.
\end{lemma}

%% file: Revision_Round_1/cumulants.tex
\subsection{Proof of Theorem \ref{thm:cumulants}}\label{subsec:cum}

The strategy to prove the first theorem is based on decomposing the $k$-point functions into combinatorial expressions that involve basically covariances of Gaussian random variables. This is made possible by our explicit knowledge of the Gaussian field which underlies $\Phi_\eps$. These covariances can be estimated using the transfer matrix (Equation \eqref{eq:transfer_cov}), whose scaling limit is well-known: it is the differential of the Laplacian Green's function (cf. \citet[Theorem 1]{kassel-wu}).

In order to compute the $k$-point function we will first make use of Feynman diagrams techniques, of which we provide a brief exposition in the appendix at the end of the present paper. In particular we will make use of Theorem \ref{thm:janson_general}.

\begin{proof}[Proof of Theorem \ref{thm:cumulants}]
Let us compute the function
\[
    Q_k\big(x^{(1)}_\eps,\dots,x^{(k)}_\eps\big) \coloneqq \E\left[\prod_{j=1}^k{\Phi_\eps\big(x^{(j)}_\eps\big)}\right] \ .
\]
From Definition \ref{def:GradGFF} of $\Phi_\eps\big(x^{(j)}_\eps\big)$ we know that
\[
    Q_k\big(x^{(1)}_\eps,\dots,x^{(k)}_\eps\big) = \sum_{i_1,\dots,i_k\in \mathcal{E}} \E\left[\prod_{j=1}^k:\!\Big(\nabla_{i_j}\Gamma\big(x^{(j)}_\eps\big)\Big)^2\!:\right] \ ,
\]
with $\mathcal E$ the canonical basis of $\R^d$. In our case we have $k$ products of the Wick product $:\!\big(\nabla_{i_j}\Gamma\big(x^{(j)}_\eps\big)\big)^2\!:$ (indexed by $j$, not $i_j$). So we can identify $Y_j$ in Theorem~\ref{thm:janson_general} with $:\!\big(\nabla_{i_j}\Gamma\big(x^{(j)}_\eps\big)\big)^2\!:$ for any $j\in[k]$, being $\xi_{j1}=\xi_{j2}=\nabla_{i_j}\Gamma\big(x^{(j)}_\eps\big)$.

Let us denote $\overline{x_{e_i}^{(j)}} \coloneqq \big(x^{(j)}, x^{(j)} + e_i\big),\, i\in[d]\,,\, j \in [k]$ (we drop the dependence on $\eps$ to ease notation). Also to make notation lighter we fix the labels $i_j$ for the moment and keep them implicit. We then define $\mathcal{U} \coloneqq \big\{\overline{x^{(1)}},\overline{x^{(1)}},\dots,\overline{x^{(k)}},\overline{x^{(k)}}\big\}$, where each copy is considered distinguishable. We also define $FD_0$ as the set of complete Feynman diagrams on $\mathcal U$ such that no edge joins $\overline{x^{(i)}}$ with (the other copy of) $\overline{x^{(i)}}$. That is, a typical edge $b$ in a Feynman diagram $\gamma$ in $FD_0$ is of the form $\big(\overline{x^{(j)}}, \overline{x^{(m)}}\big)$, with $j\neq m$ and $j,m\in[k]$. Thus by Definition \ref{def:FD} we have
\[
    \E\left[\prod_{j=1}^k:\!\left(\nabla\Gamma\big(x^{(j)}_\eps\big)\right)^2\!:\right] = \sum_{\gamma\in FD_0} \nu(\gamma) = \sum_{\gamma\in FD_0} \prod_{b\in E_\gamma} \E\left[\nabla_{b^+}\Gamma\left((b^+)^-\right)\nabla_{b^-}\Gamma\left((b^-)^-\right)\right] \ ,
\]
where $E_\gamma$ are the edges of $\gamma$ (note that the edges of $\gamma$ connect edges of $U_\eps$) and $(b^+)^-$ denotes the tail of the edge $b^+$ (analogously for $b^-$). Lemma~\ref{lem:grad_gff_cov} and Equation~\eqref{eq:transfer_cov} yield
\[
    \E\left[\prod_{j=1}^k:\!\left(\nabla\Gamma\big(x^{(j)}_\eps\big)\right)^2\!:\right] = \sum_{\gamma\in FD_0} \prod_{b\in E_\gamma} T(b^+,b^-) \ .
\]

Now we would like to express Feynman diagrams in terms of permutations. We first note that any given $\gamma\in FD_0$ cannot join $\overline{x^{(i)}}$ with itself (neither the same nor the other copy of itself). So instead of considering permutations $\sigma\in \mathrm{Perm}(\mathcal U)$ we consider permutations $\sigma'\in S_k$, being $S_k$ the group of permutations of the set $[k]$. Any $\gamma\in FD_0$ is a permutation $\sigma\in \mathrm{Perm}(\mathcal U)$, but given the constraints just mentioned, we can think of them as permutations \mbox{$\sigma'\in S_k$} without fixed points; that is, $\sigma'\in S_k^0$. Thus
\[
    \E\left[\prod_{j=1}^k:\!\left(\nabla\Gamma\big(x^{(j)}_\eps\big)\right)^2\!:\right] = \sum_{\sigma'\in S_k^0} c(\sigma') \prod_{j=1}^k T\left(\overline{x^{(j)}},\overline{x^{(\sigma'(j))}}\right) \ ,
\]
with $c(\sigma')$ a constant that takes into account the multiplicity of different permutations $\sigma$ that give rise to the same $\sigma'$, depending on its number of subcycles.

Let us disassemble this expression even more. In general $\sigma'$ can be decomposed in $q$ cycles. Since $\sigma'\in S_k^0$ (in particular, it has no fixed points), there are at most $\floor{k/2}$ cycles in a given $\sigma'$. Hence,
\[
    \E\left[\prod_{j=1}^k:\!\left(\nabla\Gamma\big(x^{(j)}_\eps\big)\right)^2\!:\right] = \sum_{q=1}^{\floor{k/2}} \sum_{\substack{\sigma'\in S_k^0\\\sigma'=\sigma'_1\dots\sigma'_q}} c(\sigma') \prod_{h=1}^q \prod_{j\in\sigma'_h} T\left(\overline{x^{(j)}},\overline{x^{(\sigma'_h(j))}}\right) \ ,
\]
where the notation $j\in\sigma_h'$ means that $j$ belongs to the domain where $\sigma_h'$ acts (non trivially). As for $c(\sigma')$, given a cycle $\sigma'_i$, $i\in[q]$, it is straightforward to see that there are $2^{|\sigma'_i|-1}$ different Feynman diagrams in $FD_0$ that give rise to $\sigma_i'$, where $|\sigma'_i|$ is the length of the orbit of $\sigma'_i$. This comes from the fact that we have two choices for each element in the domain, but swapping them gives back the original Feynman diagram, so we obtain 
\[
    c(\sigma') = \prod_{i\in[q]} 2^{|\sigma'_i|-1} \ .
\]

Now we note that a cyclic decomposition of a permutation of the set $[k]$ determines a partition $\pi\in\Pi\left([k]\right)$ (although not injectively). This way, a sum over the number of partitions $q$ and $\sigma'\in S_k^0$ with $q$ cycles can be written as a sum over partitions $\pi$ with no singletons, and a sum over full cycles in each block $B$ (that is, those permutations consisting of only one cycle). Hence
\[
    \E\left[\prod_{j=1}^k:\!\left(\nabla\Gamma\big(x^{(j)}_\eps\big)\right)^2\!:\right] = \sum_{\substack{\pi\in\Pi([k])\\\text{\normalfont{w/o singletons}}}} \prod_{B\in\pi} \sum_{\sigma\in S_\mathrm{cycl}(B)} 2^{|B|-1} \prod_{j\in B} T\left(\overline{x^{(j)}},\overline{x^{(\sigma(j))}}\right) \ ,
\]
where we also made the switch between $\prod_{B\in\pi}$ and $\sum_{\sigma\in S_\mathrm{cycl}(B)}$ by grouping by factors. Alternatively, we can express this average in terms of $S_\mathrm{cycl}^0(B)$, the set of full cycles without fixed points, as
\[
    \E\left[\prod_{j=1}^k:\!\left(\nabla\Gamma\big(x^{(j)}_\eps\big)\right)^2\!:\right] = \sum_{\pi\in\Pi([k])} \prod_{B\in\pi} \sum_{\sigma\in S_\mathrm{cycl}^0(B)} 2^{|B|-1} \prod_{j\in B} T\left(\overline{x^{(j)}},\overline{x^{(\sigma(j))}}\right) \ .
\]

% As for the multiplicity factor $c(\sigma)$, a straightforward counting calculation shows that it only depends on the size of the partitions $B$, and is equal to
% \[
%     c(\sigma) = c\left(|B|\right) = 2^{\,\abs{B}-1} \ .
% \]
Finally, we need to put back the subscript $i_j$ in the elements $\overline{x^{(j)}}$ and sum over $i_1,\dots,i_k\in \mathcal{E}$. Note that for any function $f:\mathcal E^k \to \R$ we have
\[
    \sum_{i_1,\dots,i_k\in \mathcal{E}} f(i_1,\dots,i_k) = \sum_{\eta:[k]\to\mathcal E} f\left(\eta(1),\dots,\eta(k)\right) \ ,
\]
so that
\[
    \E\left[\prod_{j=1}^k{\Phi_\eps\big(x^{(j)}_\eps\big)}\right] = \sum_{\eta:[k]\to\mathcal E} \sum_{\pi\in\Pi([k])} \prod_{B\in\pi} 2^{\,\abs{B}-1} \sum_{\sigma\in S_\mathrm{cycl}^0(B)} \prod_{j\in B} T\left(\overline{x_{\eta(j)}^{(j)}},\overline{x_{\eta(\sigma(j))}^{(\sigma(j))}}\right) \ ,
\]
and grouping the $\eta(j)$'s according to each block $B\in\pi$ we get
\[
    \E\left[\prod_{j=1}^k{\Phi_\eps\big(x^{(j)}_\eps\big)}\right] =  \sum_{\pi\in\Pi([k])} \prod_{B\in\pi} 2^{\,\abs{B}-1} \sum_{\sigma\in S_\mathrm{cycl}^0(B)} \sum_{\eta:B\to\mathcal E} \prod_{j\in B} T\left(\overline{x_{\eta(j)}^{(j)}},\overline{x_{\eta(\sigma(j))}^{(\sigma(j))}}\right) \ .
\]

Regarding the transfer matrix $T$, using Equation~\eqref{eq:transfer_cov} we can write the above expression as
\[
    \E\left[\prod_{j=1}^k{\Phi_\eps\big(x^{(j)}_\eps\big)}\right] =  \sum_{\pi\in\Pi([k])} \prod_{B\in\pi} 2^{\,\abs{B}-1} \sum_{\sigma\in S_\mathrm{cycl}^0(B)} \sum_{\eta:B\to\mathcal E} \prod_{j\in B} \nabla_{\eta(j)}^{(1)}\nabla_{\eta(\sigma(j))}^{(2)}G_{U_\eps}\big(x^{(j)}_\eps,x^{(\sigma(j))}_\eps\big) \ ,
\]
obtaining the first result of the theorem. Finally, using Lemma~\ref{lemma:conv_green_diff} we obtain the second statement.
\end{proof}

\subsection{Proof of Corollary \ref{cor:cumulants} and Proposition \ref{thm:conformal}}

\begin{proof}[Proof of Corollary \ref{cor:cumulants}]
Recall that Definition~\ref{def:cum_mult} yields
\begin{equation}\label{eq:cum_mom}
    \E\left[\prod_{j=1}^k\Phi_\eps(x_\eps^{(j)})\right] = \sum_{\pi\in\Pi([k])} \prod_{B\in\pi} \kappa\left(\Phi_\eps(x_\eps^{(j)}) : j\in [k]\right) \ .
\end{equation}
From expressions \eqref{eq:cumulants} and \eqref{eq:moments_limit} in Theorem \ref{thm:cumulants} let us see that the equality follows factor by factor by using strong induction. For $k=1$ it is trivially true since the mean of the field is $0$. Now let now us assume that it holds for $n=1,\dots,k-1$. From~\eqref{eq:mom_to_cum} we have that
\[
    \kappa\Big(\Phi_\eps\big(x^{(j)}_\eps\big):j\in[k]\Big) = \E\left[\prod_{j=1}^k\Phi_\eps\big(x^{(j)}_\eps\big)\right] - \sum_{\substack{\pi\in\Pi([k])\\|\pi|>1}} \prod_{B\in\pi} \kappa\Big(\Phi_\eps\big(x^{(j)}_\eps\big):j\in B\Big) \ .
\]
Using again~\eqref{eq:mom_to_cum} on the expectation term and the induction hypothesis, after cancellations we get
\[
    \kappa\Big(\Phi_\eps\big(x^{(j)}_\eps\big):j\in[k]\Big) = 2^{k-1}\sum_{\sigma\in S_\mathrm{cycl}^0([k])} \sum_{\eta:[k]\to\mathcal{E}}\prod_{j=1}^k \nabla_{\eta(j)}^{(1)}\nabla_{\eta(\sigma(j))}^{(2)}G_{U_\eps}\big(x^{(j)}_\eps,x^{(\sigma(j))}_\eps\big) \ .
\]
Thus the proof follows by induction.
\end{proof}

The equality in absolute value between our cumulants and those of \citet[Theorem 1]{durre} allow us to adapt his proof and conclude that, in the case of $d=2$, our field is conformally covariant with scale dimension 2.
\begin{proof}[Proof of Proposition \ref{thm:conformal}]
It is known~\cite[Proposition 1.9]{berenotes} that the continuum Green's function $G_U(\cdot, \cdot)$, defined in Equation \eqref{def:cGF}, is conformally invariant against a conformal mapping $h:U\to U'$; that is, for any $v\neq w\in U$,
\[
    G_U(v,w) = G_{U'}\left(h(v),h(w)\right) \ .
\]
Recalling expression \eqref{eq:limiting_cumulants} for the limiting cumulants we see that, for any integer $\ell \geq 2$,
\[
    \lim_{\eps\to 0} \eps^{-2\ell} \kappa\Big(\Phi^U_\eps\big(x^{(j)}_\eps\big):j\in[\ell]\Big) = 2^{\ell-1}\sum_{\sigma\in S_\mathrm{cycl}^0([\ell])} \sum_{\eta:[\ell]\to\mathcal{E}}\prod_{j=1}^\ell \partial_{\eta(j)}^{(1)}\partial_{\eta(\sigma(j))}^{(2)}G_{U'}\Big(h_\eps\big(x^{(j)}\big), h_\eps\big(x^{(\sigma(j))}\big)\Big) \ ,
\]
where the derivatives on the right hand side act on $G_{U'}\circ (h_\eps,h_\eps)$, not on $G_{U'}$.
From the cumulants expression we deduce that, for a given permutation $\sigma$ and assignment $\eta$, each point $x^{(j)}$ will appear exactly twice in the arguments of the product of differences of $G_{U'}$. Thus, using the chain rule and the Cauchy-Riemann equations, for a fixed $\sigma$ we obtain an overall factor $\prod_{j=1}^\ell \big|h'\big(x^{(j)}\big)\big|^2$
after summing over all $\eta$. We then obtain
\[
    \lim_{\eps\to 0} \eps^{-2\ell} \kappa\Big(\Phi^U_\eps\big(x^{(j)}_\eps\big):j\in[\ell]\Big) = \prod_{j=1}^\ell \abs{h'\big(x^{(j)}\big)}^2 \lim_{\eps\to 0} \eps^{-2\ell} \kappa\Big(\Phi^{U'}_\eps\Big(h_\eps\big(x^{(j)}\big)\Big):j\in[\ell]\Big) \ .
\]
The result follows plugging this expression into the moments.
\end{proof}

%% file: Revision_Round_1/tightness2.tex
\subsection{Proof of Theorem \ref{thm:goes_to_WN}}

The proof of this Theorem will be split into two parts. First we will show that the family $(\Phi_\eps)_{\eps>0}$ is tight in some appropriate Besov space and then we will show convergence of finite-dimensional distributions $(\langle \Phi_{\eps}, f_i\rangle)_{i\in [m]}$ and identify the limit. \\

\noindent
\textbf{Tightness}
\begin{prop}\label{thm:tightness}
    Let $U\subset \R^d$, $d\geq 2$. Under the scaling $\eps^{-d/2}$, the family $(\Phi_\eps)_{\eps>0}$ is tight in $\mathcal B_{p,q}^{\alpha,\mathrm{loc}}(U)$ for any $\alpha < -d/2$ and $p,q \in [1,\infty]$. The family is also tight in $\mathcal C^\alpha_\mathrm{loc}(U)$ for every $p,q\in[1,\infty]$ and $\alpha < -d/2$.
\end{prop}

Recall that the local Besov space $\mathcal B_{p,q}^{\alpha,\mathrm{loc}}(U)$  was defined in Definition \ref{def:locBes} and the local H\"older space $\mathcal C^\alpha_\mathrm{loc}(U)$ in Definition \ref{def:locHoel}.\\

\noindent
\textbf{Finite-dimensional distributions}

\begin{prop}\label{prop:convWN}
     Let $U\subset \R^d$ and $d\geq 2$. There exists a normalization constant $\chi>0$ such that, for any set of functions $\left\{f_i\in L^2(U):\, i\in[m], \, m\in\N\right\}$, the random elements $\langle\Phi_\eps,f_i\rangle$ converge in the following sense:
    \[
        \left(\frac{\eps^{-d/2}}{\sqrt \chi}\langle\Phi_\eps,f_i\rangle\right)_{i\in[m]} \overset{d}\longrightarrow \left(\langle W, f_i \rangle \right)_{i\in [m]}
    \]
    as $\eps \to 0$.
\end{prop}

\subsubsection{Proof of Proposition \ref{thm:tightness}}

We will use the tightness criterion given in Theorem 2.30 in \citet{mf}. First we need to introduce some notation. 
Let $f$ and $(g^{(i)})_{1\leq i<2^d}$ be compactly supported test functions of class $\mathcal C^r_c(\R^d)$, $r\in \mathbb{N}$. Let $\Lambda_n \coloneqq \Z^d/2^n$, and let $R>0$ be such that
\begin{equation}\label{eq:support}
    \supp{f} \subset B_0(R) \ , \quad \supp{g^{(i)}} \subset B_0(R) \ , \quad i < 2^d \ .
\end{equation}
Let $K \subset U$ be compact and $k \in \N$. We say that the pair $(K, k)$ is \emph{adapted} if
\[
    2^{-k}R < \dist{(K,U^c)} \ .
\]
We say that the set $\mathcal K$ is a \emph{spanning sequence} if it can be written as
\[
    \mathcal K = \left\{(K_n,k_n): n \in \N\right\} \ ,
\]
where $(K_n)$ is an increasing sequence of compact subsets of $U$ such that $\bigcup_n K_n = U$, and for every $n$ the pair $(K_n, k_n)$ is adapted.

\begin{theorem}[Tightness criterion,{~\citet[Theorem~2.30]{mf}}]\label{thm:criterion}
Let $f, (g^{(i)})_{1\leq i<2^d}$ in $\mathcal C^r_c(\R^d)$ with the support properties mentioned above, and fix $p\in[1,\infty)$ and $\alpha,\beta\in\R$ satisfying $\abs{\alpha}, \abs{\beta} < r, \alpha<\beta$. Let $(\Phi_m)_{m\in\N}$ be a family of random linear functionals on $\mathcal{C}_c^r(U)$, and let $\mathcal K$ be a spanning sequence. Assume that for every $(K, k) \in \mathcal K$ , there exists a constant $c = c(K,k) < \infty$ such that for every $m \in \N$,
\begin{equation}\label{eq:crit_1}
    \sup_{x\in\Lambda_k\cap K}{\E\left[\abs{\big\langle \Phi_m,f\big(2^k(\cdot-x)\big)\big\rangle}^p\right]^{1/p}} \leq c
\end{equation}
and
\begin{equation}\label{eq:crit_2}
    \sup_{x\in\Lambda_n\cap K}{2^{dn}\E\left[\abs{\big\langle \Phi_m,g^{(i)}\big(2^n(\cdot-x)\big)\big\rangle}^p\right]^{1/p}} \leq c \, 2^{-n\beta} \ , \quad i<2^d \ , \ n\geq k \ .
\end{equation}
Then the family $(\Phi_m)_m$ is tight in $\mathcal{B}^{\alpha,\mathrm{loc}}_{p,q}(U)$ for any $q\in[1,\infty]$. If moreover $\alpha < \beta -  d /p$, then the family is also tight in $\mathcal C^\alpha_\mathrm{loc}(U)$.
\end{theorem}

\begin{proof}[Proof of Proposition \ref{thm:tightness}]
We will consider an arbitrary scaling $\eps^{\gamma}, \gamma \in\R$, and then choose an optimal one to make the fields tight. We define $\widetilde\Phi_\eps$ as the scaled version of $\Phi_\eps$, that is,
\[
    \widetilde\Phi_\eps(x) \coloneqq \eps^\gamma \Phi_\eps(x) = \eps^\gamma \sum_{i=1}^d{:\!\left(\nabla_i\Gamma(x)\right)^2\!:} \ , \quad x \in U_\eps \ .
\]
The family of random linear functionals $(\Phi_m)_{m\in\N}$ in Theorem~\ref{thm:criterion} is to be identified with the fields $(\widetilde\Phi_\eps)_{\eps>0}$ taking for example $\eps$ decreasing to zero along a dyadic sequence. Now let us expand the expressions \eqref{eq:crit_1} and \eqref{eq:crit_2} in Theorem \ref{thm:criterion}. To simplify notation, let us define $f_{k,x}(\cdot) \coloneqq f\big(2^k(\cdot-x)\big)$ for $k\in\N$ and $x\in \R^d$, and analogously for $g^{(i)}$.

In the proof we will set $p\in 2\N$. This will not affect the generality of our results because of the embedding of local Besov spaces described in Lemma~\ref{lemma:embedding}.
This means that we can read \eqref{eq:crit_1} and \eqref{eq:crit_2} forgetting the absolute value in the left-hand side. Let us rewrite the $p$-th moment of $\big\langle\widetilde\Phi_\eps,f_{k,x}\big\rangle$ as
\begin{equation}\label{eq:crit_1_1}
   0\le \E\left[{\big\langle \widetilde\Phi_\eps,f_{k,x}\big\rangle}^p\right] = \eps^{\gamma p}\, \E\left[\int_{U^p}\Phi_\eps\left(\floor{x_1/\eps}\right)\cdots\Phi_\eps\left(\floor{x_p/\eps}\right) f_{k,x}(x_1)\cdots f_{k,x}(x_p)\, \d x_1\cdots\d x_p \right]\ .
\end{equation}
We will seek for a more convenient expression to work with. If we allow ourselves to slightly abuse the notation for $\widetilde\Phi_\eps$, then we can express it in a piece-wise continuous fashion as
\[
    \widetilde\Phi_\eps(x) = \eps^{\gamma} \sum_{y\in U_\eps} \1_{S_1(y)}(x) \sum_{i=1}^d{:\!\left(\nabla_i\Gamma(y)\right)^2\!:} = \eps^{\gamma} \sum_{y\in U_\eps} \1_{S_1(y)}(x) \Phi_\eps(y) \ ,
\]
where $S_a(y)$ is the square of side-length $a$ centered at $y$. Under a change of variables, if we define $U^\eps \coloneqq U\cap\eps\Z^d$ (mind the superscript and the definition which is different from that of $U_\eps$ in Section \ref{sec:not}) then
\[
    \widetilde\Phi_\eps(x) = \eps^{\gamma} \sum_{y\in U^\eps} \1_{S_\eps(y/\eps)}(x) \Phi_\eps(y/\eps) \ .
\]
This way, expression \eqref{eq:crit_1_1} now reads
\[
    \E\left[{\big\langle \widetilde\Phi_\eps,f_{k,x}\big\rangle}^p\right] = \eps^{\gamma p}\, \E\left[\sum_{y_1,\dots,y_p\in U^\eps} {\Phi_\eps(y_1/\eps)\cdots\Phi_\eps(y_p/\eps)}\prod_{j=1}^p{\int_{S_\eps(y_j)}{f_{k,x}(z) \,\d z}} \right]\ ,
\]
Therefore the left-hand side of expression \eqref{eq:crit_1} from Theorem \ref{thm:criterion} is upper-bounded by
\begin{equation}\label{eq:first}
    \eps^{\gamma} \sup_{x\in\Lambda_k \cap K}{\left[\sum_{y_1,\dots,y_p\in U^\eps} \E\left[\Phi_\eps(y_1/\eps)\cdots\Phi_\eps(y_p/\eps)\right]\prod_{j=1}^p{\int_{S_\eps(y_j)}f_{k,x}(z) \,\d z}\right]^{1/p}} \ .
\end{equation}
Analogously, expression \eqref{eq:crit_2} from Theorem \ref{thm:criterion} reads
\begin{equation}\label{eq:second}
    \eps^{\gamma} \, 2^{dn}\sup_{x\in\Lambda_n \cap K}{\left[\sum_{y_1,\dots,y_p\in U^\eps} {\E\left[\Phi_\eps(y_1/\eps)\cdots\Phi_\eps(y_p/\eps)\right]}\prod_{j=1}^p{\int_{S_\eps(y_j)}{g_{n,x}^{(i)}(z)} \,\d z}\right]^{1/p}} \ .
\end{equation}

Choose $\mathcal K = (K_n,n)_{n\in\N}$ with
\[
    K_n = \big\{x\in\R^d \mid \dist{(x,U^c)}\geq(2+\delta)R2^{-n}\big\} \ ,
\]
for some $\delta > 0$ and $R$ such that \eqref{eq:support} holds. Let us first consider \eqref{eq:second}. Given that $\supp{g^{(i)}\big(2^n(\cdot-x)\big)} \subset B_x(R2^{-n})$ we can restrict the sum over $y_j$ to the set
\[
    \Omega_{n,x} = \left\{y\in U^\eps \mid \d(y,x)<2^{-n}R+\eps\sqrt{d}/2\right\} \ .
\]
We now bound \eqref{eq:second} separately for the cases $2^n \geq R\eps^{-1}$ and $2^n < R\eps^{-1}$. If $2^n \geq R\eps^{-1}$, we have
\begin{multline*}
    \sum_{y_1,\dots,y_p\in U^\eps} {\E\left[\Phi_\eps(y_1/\eps)\cdots\Phi_\eps(y_p/\eps)\right]}\prod_{j=1}^p{\int_{S_\eps(y_j)}g_{n,x}^{(i)}(z) \,\d z} \leq \\
    \leq \sum_{y_1,\dots,y_p\in \Omega_{n,x}} {\E\left[\Phi_\eps(y_1/\eps)\cdots\Phi_\eps(y_p/\eps)\right]}\prod_{j=1}^p{\int_{S_\eps(y_j)}g_{n,x}^{(i)}(z) \,\d z} \ .
\end{multline*}
The sum over $\Omega_{n,x}$ can be bounded by a sum over a finite amount of points independent of $n$, since under the condition $2^n \geq R\eps^{-1}$ the set $\Omega_{x,n}$ has at most $3^d$ points for any $x$, $\eps$ and $n$. Let us show that the sum of these expectations is uniformly bounded by a constant.

Looking at expression \eqref{eq:cumulants2} we observe the following: any given partition $\pi\in\Pi([p])$ with no singletons can be expressed as $\pi = \{B_1,\dots,B_\ell\}$, with $1\leq \ell\leq p$ such that $\sum_{1\leq i\leq \ell} n_i = p$, with $n_i \coloneqq \abs{B_i}$. Then the cumulant corresponding to any given $B_i$ (see Corollary \ref{cor:cumulants}) is proportional to a sum over $\sigma \in S^0_\mathrm{cycl}(B_i)$ and $\eta:B_i\to\mathcal E$ of terms of the form
\[
    \prod_{j\in B_i} \nabla_{\eta(j)}^{(1)} \nabla_{\eta(\sigma(j))}^{(2)} G_{U_\eps}\big(y_j, y_{\sigma(j)}\big) \ .
\]
Using \eqref{eq:two_durre} we can bound this expression (up to a constant) by
\[
    \prod_{j\in B_i} \min{\left\{\big|y_{j} - y_{\sigma(j)}\big|^{-d},\,1\right\}} \ ,
\]
where the minimum takes care of the case in which the set $\{y_j:j\in B_i\}$ has repeated values, so that $y_j=y_{\sigma(j)}$ for some $j\in B_i$ and some $\sigma$. So we have that
\begin{equation}\label{eq:bound_sum_moments}
    \E\left[\Phi_\eps(y_1)\cdots\Phi_\eps(y_p)\right] \lesssim_{K_n} \sum_{\pi\in\Pi([p])} \prod_{B\in\pi} c(|B|) \prod_{j\in B} \min{\left\{\big|y_{j} - y_{\sigma(j)}\big|^{-d},\,1\right\}}
\end{equation}
for some constant $c(|B|)$ depending on $B$ that accounts for the sum over $\sigma\in S^0_\text{cycl}(B)$ and over $\eta:B\to\mathcal E$. Since $|y_i-y_j|\geq1$ for any $y_i,y_j\in \Omega_{n,x}$ and any $n$ and $x$, \eqref{eq:bound_sum_moments} is bounded by a constant depending only on $p$, so that
\[
    \sum_{y_1,\dots,y_p\in \Omega_{n,x}} \E\left[\Phi_\eps(y_1)\cdots\Phi_\eps(y_p)\right] \lesssim_{K_n} 3^d p\sum_{\pi\in\Pi([p])} \prod_{B\in\pi} c(|B|) \prod_{j\in B}  \lesssim_{K_n} 1 \ ,
\]
since $|\Omega_{n,x}| \leq 3^d p$ for all $n$ and $x$.

% Calling
% \[
%     M \coloneqq \max_{\substack{\eps,\,n: \\ 2^n\ge R\eps^{-1}}}\, \max_{x\in U^\eps}\,\max_{y_1,\,\ldots,\,y_p\in \Omega_{n,x}}\E\left[\left|\Phi_\eps(y_1/\eps)\cdots\Phi_\eps(y_p/\eps)\right|\right]
% \]
% we get
% \[
%     \sum_{y_1,\dots,y_p\in \Omega_{n,x}} \E\left[\left|\Phi_\eps(y_1/\eps)\cdots\Phi_\eps(y_p/\eps)\right|\right] \leq 3^d \, p M \ ,
% \]
% where we remind the reader that $p$ is fixed.

On the other hand, using the fact that
\[
    \int_{S_\eps(y_j)} \abs{g_{n,x}^{(i)}(z)} \,\d z \lesssim 2^{-dn}
\]
we obtain
\[
    \sum_{y_1,\dots,y_p\in \Omega_{n,x}} {\E\left[\Phi_\eps(y_1/\eps)\cdots\Phi_\eps(y_p/\eps)\right]}\prod_{j=1}^p{\int_{S_\eps(y_j)}{g_{n,x}^{(i)}(z) \,\d z}} \lesssim_{K_n} 2^{-dpn} \ ,
\]
which gives the bound
\[
    \eps^{\gamma} \, 2^{dn}\sup_{x\in\Lambda_n \cap K}{\left[\sum_{y_1,\dots,y_p\in U^\eps} {\E\left[\Phi_\eps(y_1/\eps)\cdots\Phi_\eps(y_p/\eps)\right]}\prod_{j=1}^p{\int_{S_\eps(y_j)}{g_{n,x}^{(i)}(z) \,\d z}}\right]^{1/p}} \lesssim_{K_n} \eps^{\gamma} \ .
\]
Observe that Theorem \ref{thm:criterion} allows the constant $c$ to depend on $K=K_n$, so the symbol $\lesssim_{K_n}$ is not an issue. Then, for any $\gamma\leq0$ we can bound the above expression by a constant multiple of $ 2^{-\gamma n}$. 
On the other hand, if $2^n < R\eps^{-1}$, we have
\[
    \int_{S_\eps(y_j)}\abs{g_{n,x}^{(i)}(z)}\d z \lesssim \eps^d \ .
\]
We also note that
\[
    \Omega_{n,x} \subset S_{\eps,x} \coloneqq \left[x - 2R2^{-n},x + 2R2^{-n}\right]^d \cap \eps\Z^d \ .
\]
Using this and calling $N\coloneqq\lfloor2R2^{-n}\eps^{-1}\rfloor$, we obtain
\[
    \sum_{y_1,\dots,y_p\in \Omega_{n,x}} {\E\left[\Phi_\eps(y_1/\eps)\cdots\Phi_\eps(y_p/\eps)\right]} \leq \sum_{y_1,\dots,y_p\in \llbracket-N,N\rrbracket^d} {\E\left[\Phi_\eps(y_1)\cdots\Phi_\eps(y_p)\right]} \ .
\]
Let us first study the behaviour of this expression for $p=2$. By Corollary \ref{cor:covariance_bound} we get
\begin{multline*}
    \sum_{y_1,y_2\in \llbracket-N,N\rrbracket^d} {\E\left[\Phi_\eps(y_1)\Phi_\eps(y_2)\right]} \lesssim_{K_n} \sum_{\substack{y_1,y_2\in \llbracket-N,N\rrbracket^d\\\mathstrut y_1 = y_2}} 1 \,+ \sum_{\substack{y_1,y_2\in \llbracket-N,N\rrbracket^d\\\mathstrut y_1 \neq y_2}} \frac{1}{\left|y_1-y_2\right|^{2d}} \\ \lesssim N^d + \sum_{y_1 \in \llbracket-N,N\rrbracket^d} \int_1^{2\sqrt2N}\frac{r^{d-1}}{r^{2d}}\d r
    = N^d + \sum_{y_1 \in \llbracket-N,N\rrbracket^d} \frac1d \left(1 - 2^{-3d/2} N^{-d}\right) \lesssim N^d \ .
\end{multline*}
Let us now analyze $\E\left[\Phi_\eps(y_1)\cdots\Phi_\eps(y_p)\right]$ for an arbitrary $p$. In the same spirit as the case $2^n \geq R \eps^{-1}$, by expression~\eqref{eq:bound_sum_moments} we know that
\begin{multline*}
    \sum_{y_1,\dots,y_p\in \llbracket-N,N\rrbracket^d} \E\left[\Phi_\eps(y_1)\cdots\Phi_\eps(y_p)\right] \lesssim_{K_n}\\ \sum_{\pi\in\Pi([p])} \prod_{B\in\pi} c(|B|) \sum_{y_1,\dots,y_p\in \llbracket-N,N\rrbracket^d} \prod_{j\in B} \min{\left\{\big|y_{j} - y_{\sigma(j)}\big|^{-d},\,1\right\}} \ .
\end{multline*}
Using Lemma~\ref{lemma:sum_p_point_function} we get
\[
    \sum_{y_1,\dots,y_p\in \llbracket-N,N\rrbracket^d} \prod_{j\in B} \min{\left\{\big|y_{j} - y_{\sigma(j)}\big|^{-d},\,1\right\}} \lesssim_{K_n} N^{\frac{n_i}{2}+d-1}
\]
by identifying $\eps$ with $1/N$. So we arrive to
\[
    \sum_{y_1,\dots,y_p\in \llbracket-N,N\rrbracket^d} \E\left[\Phi_\eps(y_1)\cdots\Phi_\eps(y_p)\right] \lesssim_{K_n} \sum_{\pi\in\Pi([p])} \prod_{B\in\pi} c(|B|) N^{\frac{n_i}{2}+d-1} \ .
\]
Now we use that
\[
    \prod_{B\in\pi} N^{\frac{n_i}{2}+d-1} = N^{(d-1)|\pi|+\frac{p}{2}} \ ,
\]
and since the sum takes place over partitions of the set $[p]$ with no singletons, putting everything back into~\eqref{eq:cumulants2} we see that the term with the largest value of $\abs{\pi}$ will dominate for large $N$. For $p$ even this happens when $\pi$ is composed of cycles of two elements, in which case $\abs{\pi} = p/2$. Hence,
\[
    \sum_{y_1,\dots,y_p \in \llbracket-N,N\rrbracket^d} \E\left[\Phi_\eps(y_1)\cdots\Phi_\eps(y_p)\right] \lesssim_{K_n} N^\frac{dp}{2}
\]
for $p$ even. Finally,
\[
    \eps^{\gamma} \, 2^{dn}\sup_{x\in\Lambda_n \cap K_n}{\left[\sum_{y_1,\dots,y_p\in U^\eps} {\E\left[\Phi_\eps(y_1/\eps)\cdots\Phi_\eps(y_p/\eps)\right]}\prod_{j=1}^p{\int_{S_\eps(y_j)}{g_{n,x}^{(i)}(z) \,\d z}}\right]^{1/p}} \lesssim_{K_n} 2^\frac{dn}{2} \eps^{\gamma+\frac{d}{2}} \ .
\]
If $\gamma \geq -d/2$ then we can bound the above expression by a constant multiple of $ 2^{\frac{dn}{2}} 2^{-\left(\gamma+\frac{d}{2}\right)n} = 2^{-\gamma n}$. Otherwise, we cannot bound it uniformly in $\eps$, as the bound depends increasingly on $\eps$ as it approaches 0.

Now we need to obtain similar bounds for \eqref{eq:crit_1}, which applied to our case takes the expression given in \eqref{eq:first}. For the case $2^n \geq R\eps^{-1}$ we have
\[
    \eps^{\gamma} \sup_{x\in\Lambda_n \cap K_n}{\left[\sum_{y_1,\dots,y_p\in U^\eps} {\E\left[\Phi_\eps(y_1/\eps)\cdots\Phi_\eps(y_p/\eps)\right]}\prod_{j=1}^p{\int_{S_\eps(y_j)}{f_{k,x}(z) \,\d z}}\right]^{1/p}} \lesssim \eps^\gamma \,2^{-dn} < \eps^{\gamma+d} \ ,
\]
which is bounded by some $c = c(K_n,n)$ whenever $\gamma \geq -d$. If $2^n < R\eps^{-1}$ instead we get
\[
    \eps^{\gamma} \sup_{x\in\Lambda_n \cap K_n}{\left[\sum_{y_1,\dots,y_p\in U^\eps} {\E\left[\Phi_\eps(y_1/\eps)\cdots\Phi_\eps(y_p/\eps)\right]}\prod_{j=1}^p{\int_{S_\eps(y_j)}{f_{k,x}(z) \,\d z}}\right]^{1/p}} \lesssim_{K_n} \eps^{\gamma+\frac{d}{2}} \,2^{-\frac{dn}{2}} \ .
\]
As before, only if $\gamma \geq -{d}/{2}$ we have the required bound.

Theorem \ref{thm:criterion} now implies that under scaling $\eps^{-d/2}$ the family $(\Phi_\eps)_{\eps>0}$ is tight in $\mathcal B_{p,q}^{\alpha,\mathrm{loc}}(U)$ for any $\alpha < -{d}/{2}$, any $q\in[1,\infty]$ and any $p \geq 2$ and even. Using Lemma \ref{lemma:embedding} this holds for any $p \in [1,\infty]$. This way, the family is also tight in $\mathcal C^\alpha_\mathrm{loc}(U)$ for every $\alpha < -{d}/{2}$.
\end{proof}

\begin{remark}
    Observe that the scaling $\eps^{-d}$ (the one used for the joint moments in Theorem \ref{thm:cumulants}) is outside the range of $\gamma$ required for the tightness bounds, and therefore it will give a trivial scaling.
\end{remark}

%% file: Revision_Round_1/wn.tex
\subsubsection{Proof of Proposition \ref{prop:convWN}}

The proof of this proposition will be divided into three parts. Firstly, we will determine the normalizing constant $\chi$ and show that it is well-defined, in the sense that it is a strictly positive finite constant. Secondly, recalling Definition~\ref{def:cum_mult}, we will demonstrate that the $n$-th cumulant $\kappa_n(\langle \Phi_{\eps}, f\rangle)$ of each random variable $\langle \Phi_{\eps}, f\rangle$, $f\in L^2(U)$, vanishes for $n\geq 3$. Finally we show that the second cumulant $\kappa_2(\langle \Phi_{\eps}, f\rangle, \langle \Phi_{\eps}, g\rangle)$, $g\in L^2(U)$, which is equal to the covariance, converges to the appropriate one corresponding to that of white noise. Once we have this, we can show that any collection $\left(\langle \Phi_{\eps}, f_1\rangle,\dots,\langle \Phi_{\eps}, f_k\rangle\right)$, $k\in\N$, is a Gaussian vector. To see this it suffices to take any linear combination $f=\sum_{i\in[k]}\alpha_i \langle \Phi_{\eps}, f_i\rangle$, $\alpha_i\in\R$ for all $i\in[k]$ so that, by multilinearity, all the cumulants $\kappa_n\left(\langle \Phi_{\eps}, f\rangle\right)$ converge to those of a centered normal with variance $\int_U f(x)^2\d x$. The ideas are partially inspired from \citet[Section 3.6]{durrethesis}.

For the rest of this Subsection we will work with test functions $f\in \mathcal C_c^\infty(U)$. The lifting of the results to every $f\in L^2(U)$ follows by a standard density argument~\cite[Chapter 1, Section 3]{janson}. Let us first derive a convenient representation of the action $\langle \Phi_{\eps}, f\rangle$ defined in Equation \eqref{def:action}. More precisely, defining $\langle \Phi_{\eps}, f\rangle_S$ as
\[
    \langle \Phi_{\eps}, f\rangle_S \coloneqq \sum_{v\in U_\eps} {f(\eps v)\Phi_\eps(v)} \ ,
\]
for any test function $f\in \mathcal C^{\infty}_c(U)$ we can write
\[
\begin{split}
\langle \Phi_{\eps}, f\rangle & = \eps^d \langle \Phi_{\eps}, f\rangle_S + R_{\eps}(f) \ ,
\end{split}
\]
where $R_{\eps}(f)$ denotes the reminder term that goes to $0$ in $L^2$, as we show in the next lemma.

\begin{lemma}
    Let $U\subset \R^d$, $d\geq 2$. For any test function $f\in\mathcal C^\infty_c(U)$ as $\eps\to0$ it holds that
    \begin{equation}\label{eq:difference_actions}
        \abs{R_{\eps}(f)} \xrightarrow{L^2} 0 \ .
    \end{equation}
\end{lemma}

\begin{proof}
Observe that
\[
    \langle\Phi_\eps,f\rangle = \int_U \Phi_\eps\left(\floor{x/\eps}\right)f(x) \d x = \sum_{x \in U_\eps}\Phi_\eps(x) \int_{A_x} f(y) \d y \ ,
\]
where $A_x \coloneqq \left\{a\in U : \floor{a/\eps}=x\right\}$. It is easy to see that $\abs{A_x} \leq \eps^d$, and given that the support of $f$ is compact and strictly contained in $U$, for $\eps$ sufficiently small (depending on $f$), the distance between this support and the boundary $\partial U$ will be larger than $\sqrt{d}\eps$. So there is no loss of generality if we assume that~$\abs{A_x} = \eps^d$.

Now, we can rewrite \eqref{eq:difference_actions} as
\begin{equation}\label{eq:I_factor}
    \abs{\sum_{x\in U_\eps}\Phi_\eps(x) \left(\int_{A_x} f(y) \d y - \eps^d \, f(\eps x)\right)} = \abs{\sum_{x\in U_\eps}\eps^d \, \Phi_\eps(x) \left(\frac{1}{|A_x|}\int_{A_x} f(y) \d y - \, f(\eps x)\right)} \ .
\end{equation}
Let us call $\mathcal I(x)$ the term
\[
\mathcal I(x) \coloneqq \frac{1}{|A_x|}\int_{A_x} f(y) \d y - \, f(\eps x) \ .
\]
The set $A_x$ is not a Euclidean ball, but it has bounded eccentricity (see \citet[Corollary 1.7]{stein2009real}). Therefore we can apply the Lebesgue differentiation theorem to claim that $\mathcal I(x)$ will be of order $o(1)$, where the rate of convergence possibly depends on $x$ and $f$.

To see statement \eqref{eq:difference_actions}, we square the expression in~\eqref{eq:I_factor} and take its expectation, obtaining
\begin{equation}\label{eq:help_eq}
    \E\left[\bigg|\sum_{x\in U_\eps}\eps^d \, \Phi_\eps(x) \,\mathcal I(x)\bigg|^2\right] \leq \eps^{2d}\, \E\left[\sum_{x\in U_\eps}\Phi_\eps^2(x)\right]\left(\sum_{x\in U_\eps}\mathcal I^2(x)\right)
\end{equation}
where we used the Cauchy-Schwarz inequality. By Corollary \ref{cor:covariance_bound} the expectation on the right-hand side can be bounded as
\[
    \E\left[\sum_{x\in U_\eps}\Phi_\eps^2(x)\right] \lesssim \sum_{x\in U_\eps} 1 = \mathcal O\big(\eps^{-d}\big)
\]
while the second term in~\eqref{eq:help_eq} is of order $o(\eps^{-d})$. With the outer factor $\eps^{2d}$~\eqref{eq:help_eq} goes to $0$, as we wanted to show.
\end{proof}

Let us remark that, by the previous lemma, proving finite-dimensional convergence of $\left\{\frac{\eps^{-d/2}}{\sqrt{\chi}}\langle \Phi_\eps,f_p\rangle:\,p\in[m]\right\}$ will be equivalent to proving finite-dimensional convergence of $\left\{\frac{\eps^{d/2}}{\sqrt{\chi}}\langle \Phi_\eps,f_p\rangle_S:\,p\in[m]\right\}$.

\subsubsection*{Definition of $\chi$}

\begin{lemma}\label{lemma:chi}
Let $G_0(\cdot, \cdot)$ be the Green's function on $\mathbb{Z}^d$ defined in Section \ref{sec:GFF}. The constant
    \[
        \chi \coloneqq 2\sum_{v\in\Z^d} \sum_{i,j\in[d]} \left(\nabla_i^{(1)}\nabla_j^{(2)}G_0(0,v)\right)^2
    \]
    is well-defined. In particular $\chi \in (8,+\infty)$.
\end{lemma}

\begin{proof}
Let us define $\kappa_0$ as
\begin{equation}\label{eq:kappa_0}
    \kappa_0(v,w) \coloneqq 2\sum_{i,j\in[d]} \left(\nabla_i^{(1)}\nabla_j^{(2)}G_0(v,w)\right)^2 \ .
\end{equation}
By translation invariance we notice that $\kappa_0(v,w) = \kappa_0(0,w-v)$. Moreover, using Lemma~\ref{lemma:lemma_29_durre}, we have that as $|v|\to+\infty$
\[
    \kappa_0(0,v)\lesssim |v|^{-2d}
\]
so that we can bound $\chi$ from above by
\[
    \chi = \sum_{v\in\Z^d} \kappa_0(0,v) \lesssim \,1 + \sum_{v\in\Z^d\setminus\{0\}} |v|^{-2d} < +\infty \ .
\]

For the lower bound, since $\kappa_0(0,v) \geq 0 $ for all $ v\in\Z^d$ we can take $v = 0$. Choosing the differentiation directions $i=j=1$ in \eqref{eq:kappa_0} we get the term $2\big(\nabla_1^{(1)}\nabla_1^{(2)}G_0(0,0)\big)^2$, which can be expressed as $8\left(G_0(0,0)-G_0(e_1,0)\right)^2$ using translation and rotation invariance of $G_0$. Now, by definition
\[
    \Delta G_0(0,0) = \frac{1}{2d} \sum_{x\in\Z^d:|x|=1}\left(G_0(x,0)-G_0(0,0)\right) = -1 \ ,
\]
from which $G_0(0,0) - G_0(e_1,0) = 1$. This implies that $\chi \geq 8$, and the lemma follows.
\end{proof}

\subsubsection*{Vanishing cumulants $\kappa_n$ for $n\geq 3$}

\begin{lemma}
For $n\geq3$, $f\in \mathcal C_c^\infty(U)$, the cumulants $\kappa_n \left(\eps^{d/2}\langle \Phi_\eps,f\rangle_S\right)$ go to $0$ as $\eps\to0$.
\end{lemma}

\begin{proof}
Recall that, by the multilinearity of cumulants, for $n\geq2$ the $n$-th cumulant satisfies
\begin{equation}\label{eq:nth_cumulant}
    \kappa_n\left(\eps^{d/2}\langle \Phi_\eps,f\rangle_S\right) = \eps^\frac{nd}{2} \sum_{v_1,\ldots,v_n \in D_\eps} \kappa\left(\Phi_\eps(v_i):i\in[n]\right) \prod_{j=1}^n f(\eps v_j) \ ,
\end{equation}
with $D\coloneqq \supp f$, which is compact inside $U$. The goal now is to show that
\[
    \eps^\frac{nd}{2} \sum_{v_1,\ldots,v_n \in D_\eps} \abs{\kappa\left(\Phi_\eps(v_i):i\in [n]\right)} \ \longrightarrow 0
\]
as $\eps\to 0$.

First, we note from the cumulants expression \eqref{eq:cumulants3} and bound \eqref{eq:two_durre} in Lemma \ref{lemma:bound_2_points_durre} that, for any set $V$ of (possibly repeated) points of $D_\eps$, with $\abs{V} = n$, we have
\[
    \abs{\kappa\left(\Phi_\eps(v):v \in V\right)} \lesssim_{D,n} \sum_{\sigma\in S^0_\mathrm{cycl}(V)}\prod_{v\in V} \min{\left\{\left|v-\sigma(v)\right|^{-d},\,1\right\}} \ .
\]
Using the above expression and Lemma \ref{lemma:sum_p_point_function}, it is immediate to see that, if $V$ has $m$ distinct points with $1\leq m\leq n$,
\[
    \sum_{\substack{v_1,\ldots,v_n \in D_\eps\\\mathstrut\text{$m$ distinct points}}} \abs{\kappa\left(\Phi_\eps(v_i):i\in [n]\right)} = \mathcal{O}_{D,n}\left(\eps^{-\frac{m}{2}-d+1}\right) \lesssim \mathcal{O}_{D,n}\left(\eps^{-\frac{n}{2}-d+1}\right)
\]
so that
\[
    \eps^\frac{nd}{2} \sum_{v_1,\ldots,v_n \in D_\eps} \abs{\kappa\left(\Phi_\eps(v_i):i\in [n]\right)} = \mathcal{O}_{D,n}\left(\eps^{\frac{1}{2}(d-1)(n-2)}\right) \ .
\]
We observe in particular that for $d\geq2$ this expression goes to $0$ for any $n\geq3$. Furthermore, going back to \eqref{eq:nth_cumulant}, since $f$ is uniformly bounded this shows that for $n\geq3$ the cumulants $\kappa_n$ go to $0$ as $\eps\to0$, as we wanted to show.
\end{proof}

\subsubsection*{Covariance structure $\kappa_2$}

\begin{lemma}
For any two functions $f_p,f_q \in \mathcal C^{\infty}_c(U)$, with $p,q\in[m]$ for $m \in\N$, we have
\[
    \eps^d\, \kappa\left(\langle\Phi_\eps,f_p\rangle_S,\langle\Phi_\eps,f_q\rangle_S\right) \xrightarrow{\eps\to0} \chi \int_U f_p(x) f_q(x) \d x \ .
\]
\end{lemma}

\begin{proof}
Without loss of generality we define the compact set $D\subset U$ as the intersection of the supports of $f_p$ and $f_q$. Then
\begin{equation}\label{eq:k_n2}
    \eps^d \,\kappa\left(\langle\Phi_\eps,f_p\rangle_S,\langle\Phi_\eps,f_q\rangle_S\right) = \eps^d \sum_{v,w \in D_\eps} f_p(\eps v) f_q(\eps w) \,\kappa\left(\Phi_\eps(v),\Phi_\eps(w)\right) \ .
\end{equation}
From Theorem \ref{thm:cumulants}, we know the exact expression of $\kappa\left(\Phi_\eps(v),\Phi_\eps(w)\right)$, given by
\begin{equation}\label{eq:2point}
    \kappa\left(\Phi_\eps(v),\Phi_\eps(w)\right) = 2\sum_{i,j\in[d]} \left(\nabla_i^{(1)}\nabla_j^{(2)}G_{U_\eps}(v,w)\right)^2 \ .
\end{equation}
Recall the constant $\kappa_0(v,w)$, defined in \eqref{eq:kappa_0}. We will approximate $\kappa\left(\Phi_\eps(v),\Phi_\eps(w)\right)$ by $\kappa_0(v,w)$ and then plug it in \eqref{eq:k_n2}. In other words, we will approximate $G_{U_\eps}(\cdot,\cdot)$ by $G_0(\cdot,\cdot)$. First we split Equation \eqref{eq:k_n2} into two parts:

\begin{multline}\label{eq:split_sum}
    \eps^d \,\kappa\left(\langle\Phi_\eps,f_p\rangle_S,\langle\Phi_\eps, f_q\rangle_S\right) = \eps^d \sum_{\substack{v,w \in D_\eps\\|v-w|\leq1/\sqrt{\eps}}} f_p(\eps v) f_q(\eps w) \,\kappa\left(\Phi_\eps(v),\Phi_\eps(w)\right) \\
    + \eps^d \sum_{\substack{v,w \in D_\eps\\|v-w|>1/\sqrt{\eps}}} f_p(\eps v) f_q(\eps w)\, \kappa\left(\Phi_\eps(v),\Phi_\eps(w)\right) \ .
\end{multline}
The second term above can be easily disregarded: remember that the cumulant for two random variables equals their covariance, so using Corollary \ref{cor:covariance_bound} we get
\[
    \eps^d \sum_{\substack{v,w \in D_\eps\\\abs{v-w}>1/\sqrt{\eps}}} f_p(\eps v) f_q(\eps w) \,\kappa\left(\Phi_\eps(v),\Phi_\eps(w)\right) \lesssim \eps^d \sum_{\substack{v,w \in D_\eps\\|v-w|>1/\sqrt{\eps}}} |v-w|^{-2d} \lesssim \sum_{\substack{z \in \Z^d\\|z|>1/\sqrt{\eps}}} |z|^{-2d} \ ,
\]
which goes to $0$ as $\eps\to0$. For the first sum in~\eqref{eq:split_sum}, let us compute the error we are committing when replacing $G_{U_\eps}$ by $G_0$. We notice that
\[
    \max_{i,j\in[d]} \sup_{v,w\in D_\eps} \sup_{\eps\in(0,\eps_D]}  \abs{\nabla_i^{(1)}\nabla_j^{(2)} G_{U_\eps}(v,w)} \leq  c_D
\]
justified by~\eqref{eq:two_durre} in Lemma \ref{lemma:bound_2_points_durre}, combined with
\[
    \max_{i,j\in[d]} \sup_{v,w\in \Z^d} \abs{\nabla_i^{(1)}\nabla_j^{(2)} G_{0}(v,w)}\le c
\]
for some $c>0$, which is a consequence of Lemma~\ref{lemma:lemma_29_durre}. Recalling that $|a^2-b^2|= |a-b||a+b|$ for any real numbers $a,b$, and setting
\[
    a\coloneqq \nabla_i^{(1)}\nabla_j^{(2)} G_{U_\eps}(v,w)
\]
and 
\[
    b\coloneqq\nabla_i^{(1)}\nabla_j^{(2)} G_0(v,w) \ ,
\]
together with \eqref{eq:one_durre} from Lemma \ref{lemma:bound_2_points_durre}, we obtain

\begin{equation*}\label{eq:approx_sum_G}
   \sum_{i,j\in[d]} \left(\nabla_i^{(1)}\nabla_j^{(2)}G_{U_\eps}(v,w)\right)^2= \sum_{i,j\in[d]} \left(\nabla_i^{(1)}\nabla_j^{(2)}G_0(v,w)\right)^2 + \mathcal{O}\big(\eps^d\big) \ .
\end{equation*}

We can use this approximation in the first summand in \eqref{eq:split_sum} and obtain
\begin{equation}\label{eq:approx}
    \eps^d \sum_{\substack{v,w \in D_\eps\\|v-w|\leq1/\sqrt{\eps}}} f_p(\eps v) f_q(\eps w)\, \kappa\left(\Phi_\eps(v),\Phi_\eps(w)\right) = \eps^d \sum_{\substack{v,w \in D_\eps\\|v-w|\leq1/\sqrt{\eps}}} f_p(\eps v) f_q(\eps w) \,\kappa_0(v,w) + \mathcal O\big(\eps^{d/2}\big) \ ,
\end{equation}
since $\abs{\left\{v,w \in D_\eps:|v-w| < 1/\sqrt{\eps}\right\}} = \mathcal O\big(\eps^{-\frac{3}{2}d}\big)$. Now, given that both $f_p$ and $f_q$ are in $\mathcal C_c^\infty(U)$, they are also Lipschitz continuous. Hence 
\[
    \eps^d \sum_{\substack{v,w \in D_\eps\\|v-w|\leq1/\sqrt{\eps}}}\abs{f_q(\eps v)-f_q(\eps w)} \abs{\kappa_0(v,w)} \lesssim \eps^d \sum_{\substack{v,w \in D_\eps\\1\leq |v-w|\leq1/\sqrt{\eps}}}\frac{\sqrt{\eps}}{|v-w|^{2d}} = o(1) \ ,
\]
so that we can replace, up to a negligible error, $f_q(\eps w)$ by $f_q(\eps v)$ in \eqref{eq:approx}, getting
\[
    \eps^d\, \kappa\left(\langle\Phi_\eps,f_p\rangle_S,\langle\Phi_\eps,f_q\rangle_S\right) = \eps^d \sum_{\substack{v,w \in D_\eps\\|v-w|\leq1/\sqrt{\eps}}} f_p(\eps v) f_q(\eps v)\, \kappa_0(v,w) +  o(1) \ .
\]
Finally the translation invariance of $\kappa_0$ implies
\[
    \lim_{\eps\to0} \,\eps^d\, \kappa\left(\langle\Phi_\eps,f_p\rangle_S,\langle\Phi_\eps,f_q\rangle_S\right) = \sum_{v\in\Z^d} \kappa_0(0,v) \int_U f_p(x) f_q(x) \d x
\]
as claimed.
\end{proof}

%% file: Revision_Round_1/discussion.tex
\section{Discussion and open questions}\label{sec:dis}

In this paper we studied properties of the gradient squared of the discrete Gaussian free field on $\mathbb{Z}^d$ such as $k$-point correlation functions, cumulants, conformal covariance in $d=2$ and the scaling limit on a domain $U\subset \R^d$. 

One of the most striking result we have obtained is the ``almost'' permanental structure of our field contrasting the block determinantal  structure of the height-one field of the Abelian sandpile studied in \citet{durre, durrethesis,kassel-wu}.
We plan to investigate implications of these structures further in the future. 

%The other direction of research we would like to put forward is the existence of an anomalous scaling for ``squared gradient fields''. It is apparent from the combination of Theorem~\ref{thm:cumulants} and Theorem~\ref{thm:goes_to_WN} that in order to obtain a limit for $\langle \Phi_\eps, \,f\rangle$ one has to ignore the way correlations decay in the gradient squared of the DGFF (namely, like $\eps^{2d}/|v-w|^{2d}$ for $v,w$ in the interior of $U$) and rather scale by the {\em sum} of all covariances in $U_\eps$, which is a factor of order $\eps^{d}$. The tightness result implies that this is the only possible non-trivial renormalization of this object. However, we conjecture that by inserting new parameters that make covariances decay more mildly we would obtain a non-trivial limiting object. 

In fact, the idea of the proof for tightness in Proposition \ref{thm:tightness} is based on the application of a criterion by~\citet{mf} for local H\"older and Besov spaces. The proof requires a precise control of the summability of $k$-point functions, which is provided by Theorem~\ref{thm:cumulants} and explicit estimates for double derivatives of the Green's function in a domain. Observe that the proof is based only on the growth of sums of moments at different points. Thus this technique can be generalized to prove tightness of other fields just by having information on these bounds, which is usually easier to obtain than the whole expression on the joint moments.

Regarding the convergence of finite-dimensional distributions, Proposition \ref{prop:convWN}, note that this strategy can be generalized to prove convergence to white noise of other families of fields, given the relatively mild conditions that we used from the field in question. Among them, one only requires knowledge on bounds of sums of joint cumulants, the existence of an infinite volume measure, and the finiteness of the susceptibility constant. Note that similar scaling results were given for random fields on the lattice satisfying the FKG inequality in~\citet{Newman1980Jan}.

\paragraph*{Acknowledgments}We would like to thank Antal J\'arai for bringing this problem to our attention, sharing his ideas with the authors and for helpful discussions on the topic. We would like to thank the participants of the workshop ``Challenges in Probability and Statistical Mechanics'' at the Technion, Haifa, and G\"unter Last for interesting comments and feedback on this work.

\section*{Funding and data availability statement}
AC initiated this work at TU Delft funded by grant 613.009.102 of the Dutch Organisation for Scientific Research (NWO). RSH was supported by a STAR cluster visitor grant during a visit to TU Delft where part of this work was carried out. AR is supported by Klein-2 grant OCENW.KLEIN.083 and did part of the work at TU Delft.

We do not analyse or generate any datasets, because our work proceeds within a theoretical and mathematical approach.

%% file: Revision_Round_1/appendix.tex
\section{Appendix: Feynman diagrams}
%%%numbering of theorems etc in Appendix
\renewcommand{\thetheorem}{A.\arabic{theorem}}
\setcounter{theorem}{0}
\renewcommand{\thedefinition}{A.\arabic{definition}}
\setcounter{definition}{0}
%%%%%%%%
When calculating expectations of products of Gaussian variables, one often obtains expressions consisting of pairwise combinations of the variables in question. It is then useful to define a graphical representation for these objects, the so-called Feynman diagrams. For a complete exposition on the subject we refer the reader to \citet[Chapters 1, 3]{janson}.

\begin{definition}[{Feynman diagrams, \citet[Definition 1.35]{janson}}]\label{def:FD}
A \emph{Feynman diagram} $\gamma$ of order $n\geq0$ and rank $r=r(\gamma)\geq0$ is a graph consisting of a set of $n$ vertices and $r$ edges without common endpoints. These are $r$ disjoint pairs of vertices, each joined by an edge, and $n-2r$ unpaired vertices. A Feynman diagram is said to be \emph{complete} if $r=n/2$ and \emph{incomplete} if $r<n/2$. Let $FD_0$ denote the set of all complete Feynman diagrams.  A Feynman diagram labeled by $n$ random variables $\xi_1,\dots,\xi_n$ defined on the same probability space is a Feynman diagram of order $n$ with vertices $1,\dots,n$, where $\xi_i$ is thought as being attached to vertex $i$. The \emph{value} $v(\gamma)$ of such a Feynman diagram $\gamma$ with edges $(i_k,j_k)$, $k=1,\dots,r$ and unpaired vertices $\{i:i\in A\}$ is given by
\[
    v(\gamma) = \prod_{k=1}^r \E\left[\xi_{i_k}\xi_{j_k}\right] \prod_{i\in A} \xi_i \ .
\]
Observe that this value is in general a random variable, and it is deterministic whenever the diagram is complete. 
\end{definition}

This definition allows us to express the expectation of the product of $n$ Gaussian random variables in terms of Feynman diagrams as follows:

\begin{theorem}[{\citet[Theorem 1.36]{janson}}]\label{thm:janson_exp}
Let $\xi_1,\dots,\xi_n$ be centered jointly normal random variables. Then
\[
    \E\left[\xi_1\cdots\xi_n\right] = \sum_\gamma v(\gamma) \ ,
\]
where the sum takes place over all $\gamma \in FD_0$ labeled by $\xi_1,\dots,\xi_n$.
\end{theorem}

We can also decompose the Wick product of $n$ Gaussian variables in terms of Feynman diagrams, as stated in the following theorem:

\begin{theorem}[{\citet[Theorem 3.4]{janson}}]
Let $\xi_1,\dots,\xi_n$ be centered jointly normal random variables. Then
\[
    :\!\xi_1\cdots\xi_n\!: \,= \sum_\gamma (-1)^{r(\gamma)} v(\gamma) \ ,
\]
being $r(\gamma)$ the rank of $\gamma$, where the sum takes place over all Feynman diagrams $\gamma$ labeled by $\xi_1,\dots,\xi_n$.
\end{theorem}

An extension of Theorem \ref{thm:janson_exp} now reads:

\begin{theorem}[{\citet[Theorem 3.8]{janson}}]\label{thm:janson_38}
Let $\xi_1,\dots,\xi_{n+m}$ be centered jointly normal random variables, with $m,n\geq0$. Then
\[
    \E\left[:\!\xi_1\cdots\xi_n\!: \xi_{n+1}\cdots\xi_{n+m}\right] = \sum_\gamma v(\gamma) \ ,
\]
where the sum takes place over all complete Feynman diagrams $\gamma$ labeled by $\xi_1,\dots,\xi_{n+m}$ such that no edge joins any pair $\xi_i$ and $\xi_j$ with $i<j\leq n$.
\end{theorem}

A formula for an even more general case can be obtained as follows:

\begin{theorem}[{\citet[Theorem 3.12]{janson}}]\label{thm:janson_general}
Let $Y_i = \;:\!\xi_{i1}\cdots\xi_{i l_i}\!:$, where $\left\{\xi_{ij}\right\}_{\substack{1\leq i\leq k, \\ 1\leq j\leq l_i}}$ are centered jointly normal variables, with $k\geq0$ and $l_1,\dots,l_k\geq 0$. Then
\[
    \E\left[Y_1\cdots Y_k\right] = \sum_\gamma v(\gamma) \ ,
\]
where we sum over all complete Feynman diagrams $\gamma$ labeled by $\left\{\xi_{ij}\right\}_{ij}$ such that no edge joins two variables $\xi_{i_1j_1}$ and $\xi_{i_2j_2}$ with $i_1=i_2$.
\end{theorem}

\begin{remark}
We said this is a formula for an even more general case than Theorem \ref{thm:janson_38} because $:\!X\!:\, = X$ for any centered normal variable.
\end{remark}

This theorem will be used for the proof of Theorem \ref{thm:cumulants}. In that case, each $Y_i$ is the Wick product of two variables, namely $Y_i = \,:\!\xi_{i1}\xi_{i2}\!:$, for all $i=1,\dots,n$. In this specific case it will hold, in fact, that $\xi_{i1} = \xi_{i2}$ for all $i$, but we keep a different notation for each variable in order to keep track of every possible Feynman diagram that can be made up from the variables $Y_i$. The value of a complete Feynman diagram $\gamma$ in this setting will be given by the expression
\[
    v(\gamma) = \prod_{s=1}^k \E\left[\xi_{\alpha_s m_{\alpha_s}}\xi_{\beta_s m_{\beta_s}}\right] \ ,
\]
with $\alpha_s,\beta_s\in[k]$, $\alpha_s \neq \beta_s$ for all $s$, and $m_{\alpha_s}, m_{\beta_s} \in \{1,2\}$.

Let us discuss a concrete example for the case $k=3$. One example is $\gamma=(V,E)$ with two copies of nodes per vertex $V=\{x_i, \tilde{x}_i: i=1,2,3\}$ and the set of undirected edges $E=\{(x_1, {x}_2), (\tilde{x}_1, x_3), (\tilde{x}_2,\tilde{x}_3)\}$ which pictorially can be depicted in Figure \ref{fig:A} below.  We have in total 8 complete Feynman diagrams in this case which can be obtained by considering the different edges resulting from pairings of the nodes $\{x_i,\tilde{x}_i: i= 1,2,3\}$ ignoring all pairings of the sort $(x_i,\tilde{x}_i)$ for all $i=1,2,3$. 

\begin{figure}[ht!]
    \centering
    \includegraphics[scale=0.8]{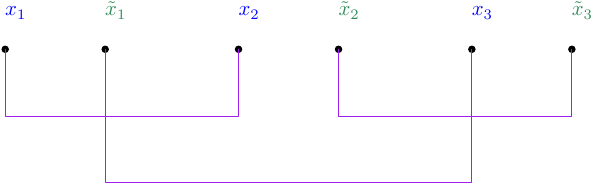}
    \caption{An example of a possible pairing of edges in a Feynman diagram.}
    \label{fig:A}
\end{figure}

%% file: main.bbl
\begin{thebibliography}{37}
\providecommand{\natexlab}[1]{#1}
\providecommand{\url}[1]{\texttt{#1}}
\expandafter\ifx\csname urlstyle\endcsname\relax
  \providecommand{\doi}[1]{doi: #1}\else
  \providecommand{\doi}{doi: \begingroup \urlstyle{rm}\Url}\fi

\bibitem[Armstrong et~al.(2017)Armstrong, Kuusi, and Mourrat]{mourratarmstrong}
S.~Armstrong, T.~Kuusi, and J.-C. Mourrat.
\newblock \emph{Quantitative Stochastic Homogenization and Large-Scale
  Regularity}.
\newblock Springer, 2017.

\bibitem[Barlow and Slade(2019)]{Barlow}
M.~T. Barlow and G.~Slade.
\newblock \emph{{Random Graphs, Phase Transitions, and the Gaussian Free
  Field}}.
\newblock Springer International Publishing, Cham, Switzerland, 2019.

\bibitem[Bauerschmidt et~al.(2014)Bauerschmidt, Brydges, and
  Slade]{bauerschmidt2014scaling}
R.~Bauerschmidt, D.~C. Brydges, and G.~Slade.
\newblock Scaling limits and critical behaviour of the $4$-dimensional
  $n$-component $|\varphi|^4$ spin model.
\newblock \emph{Journal of Statistical Physics}, 157\penalty0 (4):\penalty0
  692--742, 2014.

\bibitem[Berestycki(2015)]{berenotes}
N.~Berestycki.
\newblock Introduction to the {G}aussian free field and {L}iouville quantum
  gravity.
\newblock
  \url{https://www.math.stonybrook.edu/~bishop/classes/math638.F20/Berestycki_GFF_LQG.pdf},
  2015.
\newblock Accessed: 2022-06-30.

\bibitem[Biskup and Spohn(2011)]{BiskupSpohn}
M.~Biskup and H.~Spohn.
\newblock {Scaling limit for a class of gradient fields with nonconvex
  potentials}.
\newblock \emph{The Annals of Probability}, 39\penalty0 (1):\penalty0 224 --
  251, 2011.

\bibitem[Boutillier(2007)]{Bouti}
C.~Boutillier.
\newblock {Pattern Densities in Non-Frozen Planar Dimer Models}.
\newblock \emph{Commun. Math. Phys}, 271:\penalty0 55 -- 91, 2007.

\bibitem[Cotar et~al.(2009)Cotar, Deuschel, and M{\"u}ller]{cotar2009strict}
C.~Cotar, J.-D. Deuschel, and S.~M{\"u}ller.
\newblock Strict convexity of the free energy for a class of non-convex
  gradient models.
\newblock \emph{Communications in mathematical physics}, 286\penalty0
  (1):\penalty0 359--376, 2009.

\bibitem[Daubechies(1992)]{daubechies1992ten}
I.~Daubechies.
\newblock \emph{Ten Lectures on Wavelets}.
\newblock CBMS-NSF Regional Conference Series in Applied Mathematics. Society
  for Industrial and Applied Mathematics, 1992.

\bibitem[Ding et~al.(2012)Ding, Lee, and Peres]{Ding2012}
J.~Ding, J.~R. Lee, and Y.~Peres.
\newblock {Cover times, blanket times, and majorizing measures}.
\newblock \emph{Ann. Of Math.}, 175\penalty0 (3):\penalty0 1409--1471, 2012.

\bibitem[Ding et~al.(2021)Ding, Dubedat, and Gwynne]{Ding2021Sep}
J.~Ding, J.~Dubedat, and E.~Gwynne.
\newblock {Introduction to the Liouville quantum gravity metric}.
\newblock \emph{arXiv}, Sept. 2021.

\bibitem[D{\"u}rre(2009{\natexlab{a}})]{durre}
F.~M. D{\"u}rre.
\newblock Conformal covariance of the {A}belian sandpile height one field.
\newblock \emph{Stochastic Processes and their Applications}, 119\penalty0
  (9):\penalty0 2725--2743, 2009{\natexlab{a}}.

\bibitem[D{\"u}rre(2009{\natexlab{b}})]{durrethesis}
F.~M. D{\"u}rre.
\newblock \emph{Self-organized critical phenomena}.
\newblock PhD thesis, Ludwig-Maximilians-Universit{\"a}t M{\"u}nchen, June
  2009{\natexlab{b}}.

\bibitem[Eisenbaum and Kaspi(2009)]{Eisen}
N.~Eisenbaum and H.~Kaspi.
\newblock {On permanental processes }.
\newblock \emph{Stoch. Proc. Appl.}, 119:\penalty0 1401 -- 1415, 2009.

\bibitem[Evans(2010)]{Evans2010}
L.~C. Evans.
\newblock \emph{{Partial Differential Equations}}.
\newblock American Mathematical Society, 2010.

\bibitem[Funaki(2005)]{funaki}
T.~Funaki.
\newblock Stochastic interface models.
\newblock \emph{Lectures on Probability Theory and Statistics}, 1869:\penalty0
  103--274, 01 2005.

\bibitem[Furlan and Mourrat(2017)]{mf}
M.~Furlan and J.-C. Mourrat.
\newblock A tightness criterion for random fields, with application to the
  {I}sing model.
\newblock \emph{Electronic Journal of Probability}, 22:\penalty0 1--29, 2017.

\bibitem[Glimm and Jaffe(1987)]{Glimm}
J.~Glimm and A.~Jaffe.
\newblock \emph{{Quantum Physics}}.
\newblock Springer, New York, NY, New York, NY, USA, 1987.

\bibitem[Hairer(2014)]{Hairer}
M.~Hairer.
\newblock A theory of regularity structures.
\newblock \emph{Inventiones Mathematicae}, 198 (2):\penalty0 269--504, 2014.

\bibitem[Hough et~al.(2009)Hough, Krishnapur, Peres, and Virag]{manju}
J.~Hough, M.~Krishnapur, Y.~Peres, and B.~Virag.
\newblock \emph{Zeros of Gaussian Analytic Functions and Determinantal Point
  Processes}, volume~51 of \emph{University Lecture Series}.
\newblock American Mathematical Society, 2009.

\bibitem[Janson(1997)]{janson}
S.~Janson.
\newblock \emph{Gaussian Hilbert Spaces}.
\newblock Cambridge Tracts in Mathematics. Cambridge University Press, 1997.

\bibitem[Jerison et~al.(2014)Jerison, Levine, and Sheffield]{jerison}
D.~Jerison, L.~Levine, and S.~Sheffield.
\newblock {Internal DLA and the Gaussian free field}.
\newblock \emph{Duke Mathematical Journal}, 163\penalty0 (2):\penalty0 267 --
  308, 2014.

\bibitem[Kang and Makarov(2013)]{kang2013gaussian}
N.-G. Kang and N.~G. Makarov.
\newblock Gaussian free field and conformal field theory.
\newblock \emph{Ast{\'e}risque}, 353:\penalty0 1--136, 2013.

\bibitem[Kassel and Wu(2015)]{kassel-wu}
A.~Kassel and W.~Wu.
\newblock Transfer current and pattern fields in spanning trees.
\newblock \emph{Probability Theory and Related Fields}, 163\penalty0
  (1):\penalty0 89--121, 2015.

\bibitem[Kenyon(2001)]{Kenyon}
R.~Kenyon.
\newblock {Dominos and the Gaussian Free Field}.
\newblock \emph{The Annals of Probability}, 29\penalty0 (3):\penalty0 1128 --
  1137, 2001.

\bibitem[Last and Penrose(2017)]{Last}
G.~Last and M.~Penrose.
\newblock \emph{Lectures on the Poisson Process}.
\newblock Cambridge University Press. IMS, 2017.

\bibitem[Lawler and Limic(2010)]{lawlerlimic}
G.~Lawler and V.~Limic.
\newblock \emph{Random Walk: A Modern Introduction}.
\newblock Cambridge Studies in Advanced Mathematics. Cambridge University
  Press, 2010.

\bibitem[McCullagh and M{\o}ller(2006)]{mccullagh2006permanental}
P.~McCullagh and J.~M{\o}ller.
\newblock The permanental process.
\newblock \emph{Advances in applied probability}, 38\penalty0 (4):\penalty0
  873--888, 2006.

\bibitem[Meyer and Salinger(1992)]{meyer1992wavelets}
Y.~Meyer and D.~Salinger.
\newblock \emph{Wavelets and Operators: Volume 1}.
\newblock Cambridge Studies in Advanced Mathematics. Cambridge University
  Press, 1992.

\bibitem[Nadaf and Spencer(1997)]{NadafSpencer}
A.~Nadaf and T.~Spencer.
\newblock {On homogenization and scaling limit of some gradient perturbations
  of a massless free field}.
\newblock \emph{Commun. Math. Phys.}, 183:\penalty0 55--84, 1997.

\bibitem[Newman(1980)]{Newman1980Jan}
C.~M. Newman.
\newblock {Normal fluctuations and the FKG inequalities}.
\newblock \emph{Commun. Math. Phys.}, 74\penalty0 (2):\penalty0 119--128, Jan
  1980.

\bibitem[Schramm and Sheffield(2009)]{Schramm2009Jan}
O.~Schramm and S.~Sheffield.
\newblock {Contour lines of the two-dimensional discrete Gaussian free field}.
\newblock \emph{Acta Math.}, 202\penalty0 (1):\penalty0 21--137, Jan. 2009.

\bibitem[Sheffield(2007)]{Sheffield2007Nov}
S.~Sheffield.
\newblock {Gaussian free fields for mathematicians}.
\newblock \emph{Probab. Theory Related Fields}, 139\penalty0 (3):\penalty0
  521--541, Nov. 2007.

\bibitem[Spitzer(1964)]{spitzer}
F.~Spitzer.
\newblock \emph{Principles of Random Walk}.
\newblock Graduate texts in mathematics. Springer, 1964.

\bibitem[Stein and Shakarchi(2009)]{stein2009real}
E.~Stein and R.~Shakarchi.
\newblock \emph{Real Analysis: Measure Theory, Integration, and Hilbert
  Spaces}.
\newblock Princeton University Press, 2009.

\bibitem[Sznitman(2012)]{sznitman2012topics}
A.~S. Sznitman.
\newblock \emph{Topics in Occupation Times and Gaussian Free Fields}.
\newblock Zurich lectures in advanced mathematics. European Mathematical
  Society, 2012.

\bibitem[Velenik(2006)]{Velenik2006Jan}
Y.~Velenik.
\newblock Localization and delocalization of random interfaces.
\newblock \emph{Probability Surveys}, 3:\penalty0 112--169, 2006.

\bibitem[Wilson(2011)]{Wilson2011XORIsingLA}
D.~Wilson.
\newblock Xor-ising loops and the gaussian free field.
\newblock \emph{arXiv:1102.3782}, 2011.

\end{thebibliography}
